\documentclass[11pt]{amsart}
\usepackage{amscd}
\usepackage{amsmath}
\usepackage{amsfonts}
\usepackage{amssymb}
\usepackage{graphicx}
\usepackage{enumerate, amstext}

\newtheorem{thm}{Theorem}[section]
\newtheorem{cor}[thm]{Corollary}
\newtheorem{lem}[thm]{Lemma}

\theoremstyle{definition}
\newtheorem{defn}[thm]{Definition}

\theoremstyle{remark}

\numberwithin{equation}{section}

\def\diag{{\rm diag}\,}
\def\cA{{\mathcal A}}
\def\cB{{\mathcal B}}
\def\cI{{\mathcal I}}

\def\cS{{\mathcal S}}
\def\IC{{\mathbb C}}

\def\t{{\theta}}
\def\tr{{\rm tr}\,}

\def\cV{{\mathcal V}}

\begin{document}
\title[Maps preserving the spectrum]
{Maps preserving the spectrum of generalized Jordan product
      of operators}

\author{Jinchuan Hou, Chi-Kwong Li, and Ngai-Ching Wong}
\address[Hou]{Department of Mathematics, Taiyuan University of Technology,
Taiyuan, 030024, P. R. of China} \email{houjinchuan@tyut.edu.cn}
\address[Li] {Department of Mathematics, The College of William
\& Mary, Williamsburg, VA 13185, USA. {\rm Li is an honorary
professor of the University of Hong Kong}} \email{ckli@math.wm.edu}

\address[Wong]{Department of Applied Mathematics, National Sun Yat-sen
  University, Kaohsiung, 80424,
Taiwan. } \email{wong@math.nsysu.edu.tw}

\date{October 6, 2009; a revision submitted to LAA}

\thanks
{Research of the first, second and third author was partially
supported by NNSFC and PNSFS of China, by USA NSF and HK RCG, and by
Taiwan NSC grant (NSC96-2115-M-110-004-MY3), respectively.
The research began while the first author was
    visiting
the College of William and Mary in the Fall of 2005 supported by the
Freeman Foundation. }

\begin{abstract}

Let $\mathcal{A}_1, \cA_2$  be standard operator algebras on complex
Banach spaces $X_1, X_2$, respectively. For $k \ge 2$, let $(i_1,
\dots, i_m)$ be a sequence with terms chosen from $\{1, \dots, k\}$,
and define the generalized Jordan product
$$
T_1\circ \cdots \circ T_k\ =\ T_{i_1}\cdots T_{i_m} + T_{i_m}\cdots
T_{i_1}
$$
on elements in $\cA_i$.  This includes the usual Jordan product
$A_1\circ A_2= A_1A_2 + A_2 A_1$, and the triple $\{A_1,A_2,A_3\}=
A_1 A_2 A_3 + A_3 A_2 A_1$.   Assume that at least  one of the terms
in $(i_1, \dots, i_m)$ appears exactly once.  Let a map $\Phi
:\mathcal{A}_1\rightarrow\mathcal{A}_2$ satisfy that
$$
\sigma(\Phi(A_1)\circ \cdots \circ \Phi(A_k)) =
\sigma(A_1\circ\cdots \circ A_k)
$$
whenever any one of $A_1, \ldots, A_k$
 has rank at most one.  It is shown in this paper that if
 the range of $\Phi$ contains all operators of rank at
 most three, then $\Phi$
must be a Jordan isomorphism multiplied by an $m$th root of unity.
Similar results  for maps between self-adjoint operators acting on
Hilbert spaces are also obtained.
\end{abstract}
\maketitle

\section {Introduction }

There has been considerable interest in studying spectrum preserving
maps on operator algebras in connection to the Kaplansky's problem
on characterization of linear maps between Banach algebras
preserving invertibility; see \cite{K,JS,AM,Sourour,A2}. Early
study focus on linear maps, additive maps, or multiplicative maps;
see, e.g., \cite{LT}. Moreover, researchers considered maps
preserving different types of spectra of operators such as the
approximate spectrum, left invertible spectrum, right invertible
spectrum, etc. Despite these variations, the maps often have the
standard form
$$A \mapsto S^{-1}AS \qquad \hbox{ or }
\quad A \mapsto S^{-1}A^*S$$ for a suitable invertible operator $S$,
and $A^*$ is the dual of $A$ if $A$ is a (bounded linear) operator
between reflexive spaces. Many interesting techniques have been
developed to derive these standard forms under different settings.

\medskip
Recently, researchers have improved the results on spectrum
preserving maps by showing that the map has the standard form under
much weaker assumptions; see, e.g., \cite{ZH,WH,CH1, CH2,BH,CLS,
HH}. For example, in [12], we characterize maps $\Phi$ (not assumed to be
linear, additive or continuous) between standard
operator algebras $\mathcal{A}_1, \mathcal{A}_2$ (not necessarily
unital or closed) on complex Banach spaces $X_1, X_2$, respectively,
such that $\sigma(\Phi(A_1)*\cdots
*\Phi(A_k)) =\sigma(A_1*\cdots
*A_k)$ whenever any one of $A_i$'s is of rank at most one. Here,
$T_1* \cdots *T_k = T_{i_1}\cdots T_{i_m}$ for a sequence $(i_1,
\dots, i_m)$ with terms in $\{1, \dots, k\}$ such that one of the
terms appears exactly once. Such product covers the usual product
$T_1* \cdots *T_k = T_1\cdots T_k$, and the Jordan triple product
$T_1*T_2 = T_2T_1T_2$. It is interesting to note that we can get the
conclusion by requiring the spectrum preserving properties for low
rank operators. In particular, we do not need to consider different
types of spectra for such operators, as all of them coincide in this
case.
The list includes the left
spectrum, the right spectrum, the boundary of the spectrum, the
full spectrum, the point spectrum, the compression spectrum, the
approximate point spectrum and the surjectivity spectrum, etc.
Thus, our results in \cite{hlw}  unify and generalize
several recent results of various spectrum preservers, see, e.g., \cite{CH1, CH2}.

\medskip
In this paper, we continue this line of study. In particular, we
consider the generalized Jordan products of operators defined below.

\begin{defn}\label{Definition2.1} Fix a positive integer $k$ and a finite
sequence $(i_1,i_2, \ldots , i_m)$ such that $\{i_1, i_2,
 \ldots , i_m\}=\{1, 2, \ldots , k\}$ and there is an $i_p$ not
 equal to $i_q$ for all other $q$.
Define a  product for operators $T_1,\ldots , T_k$ by
$$
T_1\circ \cdots \circ T_k = T_{i_1} \cdots T_{i_m}+T_{i_m}\cdots
T_{i_1}.
$$
\end{defn}
Evidently, this definition covers the usual Jordan product
$T_1T_2+T_2T_1$,
    and   the  triple one:
       $\{T_1,T_2,T_3\} = T_1T_2T_3 + T_3T_2T_1$.

In the following, for $i=1,2$, let $X_i$ be a complex Banach space,
and $\mathcal{A}_i$ be a standard operator algebra on $X_i$, i.e.,
$\cA_i$ contains all continuous finite rank operators on $X_i$.
In particular, the Banach algebra ${\mathcal B}(X_i)$ of all bounded linear
operators on  $X_i$ is a standard operator algebra.
Note
that we do not assume a standard operator algebra is unital or
closed in any topology. Recall that  a Jordan isomorphism $\Phi :
\cA_1\to \cA_2$ is  either an inner automorphism or
anti-automorphism. In this case, $\sigma(\Phi(A_1)\circ\cdots
\circ\Phi(A_k)) =\sigma(A_1\circ\cdots \circ A_k)$ holds for all
$A_1, \ldots, A_k$. We will show that the converse is also true.  It
is interesting that consideration of low rank operators is again enough
 to ensure the conclusion of the
converse statement.

\begin{thm}\label{Theorem2}
Consider the product $T_1\circ \cdots \circ T_k$ defined in
Definition \ref{Definition2.1}. Suppose $\Phi:
\mathcal{A}_1\rightarrow\mathcal{A}_2$ satisfies
\begin{align}\label{eq:old2.1}
\sigma(\Phi(A_1)\circ\cdots \circ\Phi(A_k))
   =\sigma(A_1\circ\cdots \circ A_k).
\end{align}
whenever any of $A_1,\cdots, A_k$  has rank at most $1$. Suppose
also that the range of $\Phi$ contains all operators in $\cA_2$ of
rank at most $3$. Then one of the following conditions holds.
\begin{enumerate}[(1)]
    \item There exist a scalar $\lambda$ with $\lambda^{m}=1$ and
an invertible operator $T$  in $\mathcal{B}(X_1,X_2)$ such that
$$
\Phi(A)=\lambda TAT^{-1} \quad \text{for all $A$ in
$\mathcal{A}_1$}.
$$
    \item The spaces $X_1$ and $X_2$ are reflexive, and there exist
a scalar $\lambda$ with $\lambda^{m}=1$ and an invertible operator
$T \in\mathcal{B}(X_1^{\ast},X_2)$ such that
$$
\Phi(A)=\lambda TA^{\ast}T^{-1}\quad\text{for all $A$ in
$\mathcal{A}_1$}.
$$
\end{enumerate}
\end{thm}

We remark that if the condition (1) or (2) in Theorem \ref{Theorem2}
holds, then $\Phi$ satisfies \eqref{eq:old2.1} for all $A_1,\dots,
A_k$ in $\cA_1$. In fact, $\Phi$ preserves different kinds of
spectra of $A_1\circ \cdots \circ A_k$.  For the generalized Jordan products of rank
at most two appearing in \eqref{eq:old2.1}, all such kinds of
spectra coincide, however. So our results do unify, strengthen, and
generalize several theorems in literature. See, e.g., \cite[Remark
3.3]{hlw}.  Remark also that the linearity and continuity of $\Phi$
are parts of the conclusion. The proof of Theorem \ref{Theorem2} is
given in Section 3.

We also have a version for maps between the Jordan algebras of
self-adjoint operators on Hilbert spaces, given in Section 4.

We note  that our results are new even for the classical Jordan
product $AB + BA$ and triple $ABC +CBA$.
Similar to other papers, a crucial step in our proof is to show that
the map $\Phi$ actually preserves rank one operators.  To this end,
we provide some new characterizations of rank one operators in term
of the spectra of their Jordan products with rank one operators in
Section 2. Nonetheless, the technique we employ in this paper is
quite a bit different from those we usually see in the literature,
e.g., \cite{JS,H2,ZH,WH,BS3,CH1, CH2,BH,CLS, HH}.

Finally, we would like to thank the Referee
for his/her careful reading and helpful comments.

\section{Characterizations of rank one operators}

\begin{lem}\label{Lemma2.3}
Suppose $r$ and $s$ are integers such that $s > r > 0$. Let $A$ be a
nonzero operator  on a complex Banach space $X$ of dimension at
least three. The following conditions are equivalent.
\begin{enumerate}[(a)]
\item $A$ has rank one.

\item $\sigma(B^rAB^s + B^sAB^r)$ has at most two distinct nonzero
eigenvalues for any  $B$ in ${\mathcal B}(X)$.

\item  There does not exist an operator $B$ with rank at most
three such that $B^rAB^s + B^sAB^r$ has rank three and three
distinct nonzero eigenvalues.
\end{enumerate}
\end{lem}
\begin{proof}
The implications (a) $\Rightarrow$ (b) $\Rightarrow$ (c) are clear.

To prove (c) $\Rightarrow$ (a), we consider the contrapositive.
Suppose (a) is not true, i.e., $A$ has rank at least 2.

If $A$ has rank at least 3,  then there are $x_1, x_2, x_3 \in X$
such that $\{Ax_1, Ax_2, Ax_3\}$ is linearly independent. Consider
the operator matrix of $A$ on  the span of $\{x_1, x_2, x_3, Ax_1,
Ax_2, Ax_3\}$ and its complement:
$$
\left(\begin{array}{cc} A_{11} & A_{12} \\
A_{21} & A_{22}\end{array}\right).
$$ Then $A_{11} \in M_n$
with $3 \le n \le 6$. By \cite[Lemma 2.3]{hlw}, there is a
nonsingular $U$ on the span of $\{x_1, x_2, x_3, Ax_1, Ax_2, Ax_3\}$
such that $U^{-1}A_{11}U$ has an invertible 3-by-3 leading
submatrix. We may further assume that the 3-by-3 matrix is in
triangular form with nonzero diagonal entries $a_1, a_2, a_3$. Now
let $B$ in $\cA$ have operator matrix
$$\left(\begin{array}{cc}B_{11} & 0 \\ 0 & 0 \end{array}\right),$$
where $UB_{11}U^{-1} = \diag(1, b_2, b_3) \oplus 0_{n-3}$ with
$B_{11}$ using the same basis as that of $A_{11}$ and $b_2, b_3$
being chosen such that $a_1, a_2b_2^{r+s}, a_3b_3^{r+s}$ are three
distinct nonzero numbers. It follows that $B^rAB^s+B^sAB^r$ has rank
$3$ with three distinct nonzero eigenvalues.

Next, suppose $A$ has rank 2. Choosing a suitable space
decomposition of $X$, we may assume that $A$ has operator matrix
$A_1 \oplus 0$, where $A_1$ has one of the following form.
$$
{\rm (i)} \left(\begin{array}{ccc} a & 0 & b \\ 0 & 0  & 0 \\
0 & 0 & c\end{array}\right), \quad {\rm (ii)}
\left(\begin{array}{ccc}a & 0 & 0 \\ 0 & 0 &  1 \\ 0 & 0 & 0
\end{array}\right), \quad {\rm (iii)} \left(\begin{array}{ccc}  0
& 1 & 0 \\ 0 & 0 & 1 \\ 0 & 0 & 0
\end{array}\right),
\quad {\rm (iv)} \left(\begin{array}{cc} 0_2&I_2\\ 0_2&
0_2\end{array}\right).
$$

If (i) holds, set $\t = \pi/s$. Then $\cos r \t \ne \pm 1$ and $\cos
r\t\not=\pm \sqrt{\cos 2r\t}$. Let $d > 0$ such that $a(\cos r \t
\pm \sqrt{\cos 2r\t}), -2cd^{r+s}$ are three distinct nonzero
numbers. Let $B \in \cA$  be represented by the operator matrix
$$
\left(\begin{array}{ccc} \cos \t & -\sin \t & 0\\
\sin \t & \cos \t & 0 \cr 0 & 0 & d
\end{array}\right) \oplus 0.
$$
Then $B^s = -I_2\oplus [d^s]\oplus 0$, and $-(B^rAB^s+B^sAB^r)$  has
operator matrix
$$
\left(\begin{array}{ccc}
2a\cos r \t & -a\sin r \t & * \\
a \sin r \t & 0 & * \\ 0 & 0 & -2cd^{r+s}
\end{array}\right) \oplus 0,$$
which has rank $3$ with three distinct   nonzero eigenvalues $a(\cos
r \t \pm \sqrt{\cos 2r\t}), -2cd^{r+s}$.

Suppose (ii) holds. Let $d > 0$ be such that $2ad^{r+s},\ s+r \pm
2\sqrt{rs}$ are three distinct nonzero numbers. Then construct $B$
by the  operator matrix
$$
\left(\begin{array}{ccc} d & 0 & 0 \\ 0 & 1 & 0 \\ 0 & 1 & 1
\end{array}\right) \oplus 0.
$$
Then  $B^rAB^s+B^sAB^r$ has operator matrix
$$
\left(\begin{array}{ccc}
2ad^{r+s} & 0 & 0 \\ 0 & s+r & 2 \\
0 & 2rs & s+r \end{array}\right) \oplus 0,
$$
which has rank $3$ with three distinct nonzero
 eigenvalues $2ad^{r+s}, s+r
\pm 2\sqrt{rs}$.

Suppose (iii) holds. First, assume that $s = 2r$. Let $B$ be such
that $B^r$ has operator matrix
$$\left(
\begin{array}{ccc} 0 & 1 & 0 \\ 0 & 0 & 1 \\
1 & 0 & 0 \end{array}\right) \oplus 0.$$ Then $B^rAB^s+B^sAB^r $ has
operator matrix
$$\left(
\begin{array}{ccc} 0 & 1 & 0 \\ 0 & 0 & 1 \\ 2 & 0 & 0 \end{array}\right)
\oplus 0,$$ which has rank $3$ with three distinct nonzero
eigenvalues: $2^{1/3}, 2^{1/3} e^{i2\pi/3}, 2^{1/3} e^{i4\pi/3}$.

Next, suppose $s/r \ne 2$. Then $s > 2$ and $2r/s$ is not an
integer. Let $\t_1 = 2\pi/s$, $\t_2 = 4\pi/s$. Then $1, e^{ir \t_1},
e^{ir \t_2}$ are distinct because $e^{i4\pi r/s} = e^{i 2\pi (2r/s)}
\ne 1$ and $e^{ir \t_1} = e^{ir \t_2}/e^{ir \t_1} = e^{i2\pi r/s}
\ne 1$. Thus, there exists an invertible $S \in M_3$ such that
$$\left(
\begin{array}{ccc}1 & 0  & 0 \\ 0 & e^{ir \t_1} & 0\\ 0 & 0 &
e^{ir \t_2}\end{array}\right) = S^{-1}\left(
\begin{array}{ccc}1 & 0  & 0 \\
1 & e^{ir \t_1} & 0\\
0 & 2 & e^{ir \t_2}\end{array}\right)S.$$ Let $B$ have operator
matrix
$$
S\left(
\begin{array}{ccc}1 & 0  & 0 \\ 0 &
e^{i\t_1} & 0\\ 0 & 0 & e^{i\t_2}\end{array}\right)S^{-1} \oplus 0.
$$
The operator matrix $B^s = I_3 \oplus 0$ and the operator matrix of
$B^r$ has the form
$$
S\left(
\begin{array}{ccc}
1 & 0  & 0 \\
0 & e^{ir \t_1} & 0\\
0 & 0 & e^{ir \t_2}\end{array}\right)S^{-1} \oplus 0 = \left(
\begin{array}{ccc}
1 & 0  & 0 \\
1 & e^{ir \t_1} & 0\\
0 & 2 & e^{ir \t_2}\end{array}\right) \oplus 0.
$$
Then $B^rAB^s + B^sAB^r = AB^r + B^rA$ has operator matrix
$$\left(
\begin{array}{ccc} 1 & 1+e^{ir \t_1} & 0 \\
0 & 3 & e^{ir \t_1}+e^{ir \t_2} \\
0 & 0 & 2 \end{array}\right) \oplus 0,$$ which has rank $3$ with
three distinct nonzero eigenvalues.

If (iv) holds, then $X$ has dimension at least 4. We may use a
different decomposition of $X$ and assume that $A$ has operator
matrix
$$\left(
\begin{array}{rr}
0 & 1  \\
0 & 0 \end{array}\right) \oplus \left(\begin{array}{rr}
1 & 1 \\
 -1 & -1 \end{array}\right) \oplus 0.
$$
Let $\t = \pi/(2(r+s))$ and $d > 0$ be such that $1 \pm
\sqrt{\sin(2r\t)\sin(2s \t)}$ and $d^{r+s}$ are 3 distinct nonzero
numbers, and let $B$ be an operator in ${\mathcal B}(X)$ such that
$B^\ell$ has operator matrix
$$B^\ell = \left(
\begin{array}{ccc} \cos \ell\t & -\sin \ell \t & 0 \\
\sin \ell \t & \cos \ell \t & 0 \\
0 & 0 & d^\ell \end{array}\right) \oplus 0$$ for any positive
integer $\ell$. Then $B^rAB^s + B^sAB^r$ has operator matrix
$$\left(
\begin{array}{ccc} \sin((r+s)\t) & 2\cos r\t \cos s \t & 0 \\
2\sin r\t \sin s \t & \sin((r+s)\t) & 0 \\
0 & 0 & d^{r+s} \end{array} \right) \oplus 0,$$ which has rank $3$
with three  distinct nonzero eigenvalues.
\end{proof}

\begin{lem}\label{Lemma2.4}
Suppose $s$ is a positive integer.
 Let $X$ be a complex Banach space of dimension at least three.
 Let $A\in {\mathcal B}(X)$ be such that   $A^2\neq0$. Then the following are equivalent.
\begin{enumerate}[(a)]
    \item  $A$ has rank one.
    \item $\sigma(AB^s+B^sA)$ has at most two distinct nonzero
eigenvalues whenever ${\rm rank }(B)\leq 3$ and ${\rm rank}
(AB^s+B^sA)\leq 3$.
\end{enumerate}
\end{lem}
\begin{proof}   One direction is trivial.
Suppose $A$ has rank at least 2 such that $A^2 \ne 0$. First assume
that $A$ has rank 2. Choosing a suitable decomposition of $X$, we
may assume that $A$ has operator matrix $A_1 \oplus 0$, where $A_1$
has one of the following form
$${\rm (i)} \left(\begin{array}{ccc} a & 0 & b \\ 0 & 0  & 0 \\
0 & 0 & c\end{array}\right),\quad {\rm (ii)}
\left(\begin{array}{ccc}a & 0 & 0 \\ 0 & 0 &  1 \\ 0 & 0 & 0
\end{array}\right), \quad {\rm (iii)} \left(\begin{array}{ccc}  0
& 1 & 0 \\ 0 & 0 & 1 \\ 0 & 0 & 0
\end{array}\right),
\quad {\rm (iv)} \left(\begin{array}{cc} 0_2&I_2\\ 0_2&
0_2\end{array}\right).$$ Since $A^2 \ne 0$, (iv) is impossible. If
(i) holds, set $\t = \pi/(2s+1)$ so that $\cos s \t \ne \pm
\sqrt{\cos 2s\t}$. Let $d > 0$ such that $a(\cos s \t \pm \sqrt{\cos
2s\t}), 2cd^{s}$ are three distinct nonzero numbers. Let $B$ have
operator matrix
$$
\left(\begin{array}{ccc} \cos \t & -\sin \t & 0\\
\sin \t & \cos \t & 0 \cr 0 & 0 & d \end{array}\right) \oplus 0.$$
Then similar to the proof of Lemma \ref{Lemma2.3}, we see that $AB^s
+ B^sA$ has operator matrix
$$
\left(\begin{array}{ccc}
2a\cos s \t & -a\sin s \t & * \\
a \sin s \t & 0 & * \\ 0 & 0 & 2cd^{s}
\end{array}\right) \oplus 0,$$
 which has rank $3$ with three distinct nonzero eigenvalues $a(\cos s
\t \pm \sqrt{\cos 2s\t}), 2cd^{r+s}$.

Suppose (ii) holds. Let $d > 0$ be such that $2d, d \pm
\sqrt{a^2+d^2}$ are three distinct nonzero numbers. Since the matrix
$$
C = \left(\begin{array}{ccc} 0 & 1 & 0 \\ 1 & 0  & 0 \\ 0 & 2d & 2
\end{array}\right)$$ is similar to a matrix with distinct
eigenvalues $-1, 1, 2$, there exists an operator $B$ of rank 3 such
that the operator matrix of $B^s$ equals $C \oplus 0$. It follows
that the operator matrix of $AB^s+B^sA$ is
$$
\left(\begin{array}{ccc} 0 & a & 1 \\ a & 2d & 2 \\ 0 & 0 & 2d
\end{array}\right) \oplus 0,$$
which has rank 3 and distinct nonzero eigenvalues $2d, d \pm
\sqrt{a^2+d^2}$.

Suppose (iii) holds. Since the matrix
$$
C = \left(\begin{array}{ccc} 0 & 0 & 0 \\ 1 & 1  & 0 \\ 0 & 2 & 2
\end{array}\right)
$$
has distinct eigenvalues $0, 1, 2$, there exists an operator $B$ of
rank 2 such that the operator matrix of $B^s$ equals $C \oplus 0$.
Then $AB^s+B^sA$ has operator matrix
$$
\left(\begin{array}{ccc} 1 & 1 & 0 \\ 0 &3 & 3 \\ 0 & 0 & 2
\end{array}\right) \oplus 0,$$
which has rank $3$ with three distinct nonzero eigenvalues $1,2,3$.

Now, suppose $A$ has rank at least 3. Since $A^2 \ne 0$, there is $x
\in X$ such that $A^2x \ne 0$. We consider 3 cases.

Case 1. There is $x \in X$ such that $[x,Ax,A^2x]$ has dimension 3.
Decompose $X$ into $[x,Ax,A^2x]$ and its complement. The operator
matrix of $A$ has the form
$$
\left(
\begin{array}{cccc}
0 & 0 & c_1 & * \\
1 & 0 & c_2 & * \\
0 & 1 & c_3 & * \\
0 & 0 & *  & * \end{array}\right).
$$
Note that for $t > 0$, the matrix
$$C = \left(
\begin{array}{ccc}
2t & 1 & 0 \\
0 & t & 2 \\
0 & 0 & 0 \end{array}\right)$$ has three distinct eigenvalues: $2t,
t, 0$. So, there is $B_1$ of rank 2 such that $B_1^s = C$.  Let $B$
have operator matrix $B_1 \oplus 0$. Then $AB^s + B^sA$ has operator
matrix $\left(
\begin{array}{cc}
t R_1 + R_2 & * \\
0 & 0 \end{array}\right),$ where
$$R_1 = \left(
\begin{array}{ccc}
   0 & 0 & 2c_1 \\
   3 & 0 & c_2  \\
   0 & 1 & 0   \end{array}\right)
\qquad \hbox{ and } \qquad R_2 = \left(
\begin{array}{ccc}
1 & 0 & c_2  \\
0 & 3 & 2c_2  \\
0 & 0 & 2 \end{array}\right).$$ Since  $R_2$ has  distinct
eigenvalues $1,2,3$,  the matrix $tR_1 + R_2$ will have three
distinct nonzero eigenvalues for sufficiently small $t$. Hence,
$AB^s + B^sA$ has rank $3$ with three  distinct nonzero eigenvalues.

Case 2. Suppose Case 1 does not hold, and there is $x \in X$ such
that $A^2x \ne 0$ and $[x,Ax,A^2x]$ has dimension 2. Clearly, we
cannot have $Ax = \lambda x$. Otherwise, $[x,Ax,A^2x]$ has dimension
1.  Hence, $A^2x = b_1 x + b_2Ax$ so that $(b_1,b_2) \ne (0,0)$.
Since $A$ has rank at least three, there is $y \in X$ such that
$Ay\notin [x,Ax]$. We claim that there is a decomposition of $X$ so
that $A$ has operator matrix
\begin{align}\label{eq:old2.2}
\left(
\begin{array}{cc}
A_0& * \\
  0 & * \end{array}\right),
\end{align}
where $A_0\in M_3$ is in upper triangular form of rank at least $2$
and with at least one nonzero eigenvalue.

To prove our claim, suppose $Ay = c_1x+c_2Ax+c_3y$ with $c_3 \ne 0$.
Using $[x,Ax,y]$ and its complement, the operator matrix of $A$ has
the form
$$
\left(
\begin{array}{cc}
A_1& * \\
  0 & * \end{array}\right)
\qquad \hbox{ with }  \qquad A_1 = \left(
\begin{array}{ccc}
  0 & b_1 & c_1 \\
  1 & b_2 & c_2 \\
  0 &   0 & c_3  \end{array}\right),
$$
where $A_1$ has rank at least 2. Since $(b_1,b_2) \ne (0,0)$, the
matrix $A_1$ has at least two nonzero  eigenvalues including $c_3$.
We may replace $\{x,Ax,y\}$ by a linearly independent family $\{\hat
x_1, \hat x_2, \hat x_3\}$ in $[x,Ax,y]$ so that the operator matrix
of $A$ has the form described in \eqref{eq:old2.2}.

Next, suppose $Ay \notin [x,Ax,y]$. Note that $[y,Ay,A^2y]$ has
dimension 2 by our assumption in Case 2. In this subcase, $Ay \ne
\lambda y$. So, $A^2y = d_1y + d_2Ay$ with $(d_1,d_2)\neq(0,0)$.
With respect to $[x,Ax,y,Ay]$ and its complement in $X$, the
operator matrix of $A$ has the form
$$
\left(
\begin{array}{cc}
A_2& * \\
  0 & * \end{array}\right) \qquad \hbox{ with } \qquad
A_2 = \left(
\begin{array}{cc}
0 & b_1  \\
1 & b_2 \end{array}\right) \oplus \left(\begin{array}{cc}
0 & d_1 \\
1 & d_2 \end{array}\right).
$$
Since $(b_1,b_2) \ne (0,0)$ and $(d_1,d_2) \ne (0,0)$, $A_2$ has
rank at least 2 and at least 2 nonzero eigenvalues. We may choose an
independent family $\{\tilde x_1, \tilde x_2, \tilde x_3, \tilde
x_4\}$ in $[x,Ax,y,Ay]$ so that the operator matrix of $A_2$ with
respect to $[\tilde x_1, \tilde x_2, \tilde x_3, \tilde x_4]$ is in
upper triangular form, whose leading 3-by-3 submatrix $A_0$ has rank
at least $2$ and has at least one nonzero eigenvalue. So, the
operator matrix of $A$ with respect to $[\tilde x_1, \tilde x_2,
\tilde x_3]$ and its complement has the form described in (2.2). So,
our claim is verified.

Now, if $A_0$ in (2.2) is invertible, then there is $B$ with
operator matrix $B_1 \oplus 0$, where $B_1 = \diag(1, b_2, b_3)$,
and $AB^s + B^sA$ has operator matrix
$$
\left(
\begin{array}{cc}
A_0B_1^s + B_1^sA_0 & * \\
  0 & 0 \end{array}\right),$$
which has rank $3$ with three distinct nonzero eigenvalues. Suppose
$A_0$ is singular. Since $A_0$ in (2.2) has rank   two and at least
one nonzero eigenvalue, we may assume that $A_0$ has the forms
$$\left(\begin{array}{ccc} a & 0 & b \\ 0 & 0  & 0 \\
0 & 0 & c\end{array}\right) \qquad \hbox{ or }   \qquad
\left(\begin{array}{ccc}a & 0 & 0 \\ 0 & 0 &  1 \\ 0 & 0 & 0
\end{array}\right).$$
In each case, we can use the arguments in the proof when $A$ has
rank 2 to choose $B$ with operator matrix $B_1 \oplus 0$ so that
$B_1 \in M_3$ and $AB^s + B^sA$ has operator matrix
$$
\left(
\begin{array}{cc}
A_0B_1^s + B_1^sA_0 & * \\
  0 & 0 \end{array}\right),$$
which is a rank $3$ operator with three distinct nonzero
eigenvalues.

Case 3. Suppose $[x,Ax,A^2x]$ has dimension one for any nonzero $x$
in $X$. Then $A$ is a scalar operator. Let $B$ have operator matrix
$\diag(1,2,3) \oplus 0$. Then $AB^s+B^sA$ has rank $3$ and three
distinct nonzero eigenvalues. \end{proof}

\begin{cor}\label{Corollary2.5}
Suppose $s$ is a positive integer. Let $X$ be a complex Banach space
$X$ of dimension at least three, and let $A$ in ${\mathcal B}(X)$ be
nonzero. The following conditions are equivalent.
\begin{enumerate}[(a)]
    \item $A$ has rank one, or $A$ has rank two such that $A^2 =
0$.

    \item $\sigma(AB^s + B^sA)$ has at most two distinct nonzero
eigenvalues for any $B$ in ${\mathcal B}(X)$.

    \item  There does not exist an operator $B$ with rank at most
three such that $AB^s + B^sA$ has rank at most six with three
distinct nonzero eigenvalues.
\end{enumerate}
\end{cor}

\begin{proof}   (a) $\Rightarrow$ (b). If $A$ has rank one, then (b)
clearly holds. If $A$ has rank two and $A^2 = 0$, then there is a
decomposition of $X$ such that $A$ has operator matrix
$$\left(\begin{array}{ccc} 0_2 & I_2 & 0 \\ 0_2 & 0_2  & 0 \\
0 & 0 & 0\end{array}\right).$$ So, for any $B$ in $\cA$ such that
$B^s$ has operator matrix
$$\left(\begin{array}{ccc} B_{11} & B_{12} & B_{13} \\
B_{21} & B_{22}  & B_{23} \\
B_{31} & B_{32} & B_{33} \end{array}\right),$$ $AB^s+B^sA$ has
operator matrix
$$\left(\begin{array}{ccc} B_{21} & B_{22} + B_{11} & B_{23} \\
0 & B_{21} & 0 \\
0 & B_{31} & 0 \end{array}\right),$$ whose nonzero eigenvalues are
the same as those of $B_{21} \in M_2$. Thus, there are at most two
nonzero distinct eigenvalues.

The implication (b) $\Rightarrow$ (c) is clear.

Finally, we verify the implication (c) $\Rightarrow$ (a). If (c) holds, by Lemma
\ref{Lemma2.4}, we see that $A$ is either rank 1 or $A^2 = 0$. If
$A^2 = 0$, we claim that $A$ has rank at most 2. If it is not true,
then we can find $x_1, x_2, x_3$ in $X$  such that $\{Ax_1, Ax_2,
Ax_3\}$ is linearly independent. Then with respect to
$[x_1,x_2,x_3,Ax_1,Ax_2,Ax_3]$ and its complement, the operator
matrix of $A$ has the form
$$\left(\begin{array}{ccc}
0_3 & 0_3 & *  \\
I_3 & 0_3 & *  \\
0 & 0 & *  \end{array}\right).
$$
Let $B \in \cB(X)$  have rank $3$ with three distinct nonzero
eigenvalues such that $B^s$ has operator matrix
$$\left(\begin{array}{cc}
D & D  \\
0_3 & 0_3 \end{array}\right) \oplus 0, \qquad \hbox{ with } \qquad D
= \diag(1,2,3).$$ Then $AB^s + B^sA$ has rank 6 and 3 distinct
eigenvalues. Our conclusion follows.
\end{proof}

%\section{The finite dimensional cases}

\section{Maps preserving spectrum of generalized Jordan products of low rank}

Theorem \ref{Theorem2} clearly follows from the special case below,
by considering $A_{i_p}=A$ and all other $A_{i_q} = B$.

\begin{thm}\label{Theorem2.2}
Suppose a map $\Phi:\cA_1\to\cA_2$ between standard operator
algebras satisfies
\begin{align}\label{eq:2.1}
\sigma( \Phi(B)^r\Phi(A)\Phi(B)^s +  \Phi(B)^s\Phi(A)\Phi(B)^r)
   =\sigma(B^rAB^s+B^sAB^r),
\end{align}
whenever $A$ or $B$ has rank at most one. Suppose also that the
range of $\Phi$ contains all operators in $\cA_2$ of rank at most
$3$. Then one of the two assertions in Theorem \ref{Theorem2} holds
with  $m = r+s+1$.
\end{thm}

We note that the case when $s=r>0$ has been verified in \cite{hlw}.
So, unless specified otherwise, we will assume $s > r \geq 0$ in the
rest of this section. In below, we first show that $\Phi$ in Theorem
\ref{Theorem2.2} is injective.

For a Banach space $X$ denote by ${\mathcal I}_1(X)$ the set of all
rank one idempotent operators in ${\mathcal B}(X)$.  In other words,
$\cI_1(X)$ consists of all bounded operators $x\otimes f$ with $x\in
X$, $f\in X^*$ and $\left<x,f\right>=f(x)=1$.

\begin{lem}\label{lem:dependence}
Let $A, A'\in \cB(X)$ for some Banach space $X$. Suppose
$$
\left<Ax,f\right>=0 \quad\text{if and only if}\quad
\left<A'x,f\right>=0, \quad\forall x\otimes f\in \cI_1(X).
$$
Then $A'=\lambda A$ for some scalar $\lambda$.
\end{lem}
\begin{proof}
First suppose there is a nonzero $x$ in $X$ such that $Ax=\alpha x$
for some nonzero scalar $\alpha$.  Then for any $f$ in $X^*$ with
$\left<x,f\right>\neq 0$, we have $\left<Ax,f\right>\neq 0$, and
thus $\left<A'x,f\right>\neq 0$.  Hence, $A'x =\beta x$ for some
nonzero scalar $\beta$, and $Ax,A'x$ are linearly dependent.

Then suppose $\{x, Ax\}$ is linearly independent.  Choose any
$x\otimes f$ in $\cI_1(X)$ with $\left<Ax,f\right>=0$. Then for any
$g$ in $X^*$ with $\left<x,g\right>=0$, we have
$\left<x,f+g\right>=1$. If $\left<Ax,g\right>=0$ then
$\left<Ax,f+g\right>=0$, and thus $\left<A'x,f+g\right>=0$. This
eventually gives $\left<A'x,g\right>=0$.  Thus, together with the
assumption, we see that $Ax, A'x$ are linearly dependent again.

If $A$ has rank one then the assertion is plain. Assume $Ax, Ay$ are
linearly independent for some $x,y$ in $X$. Then $A'x = \lambda_x
Ax$, $A'y = \lambda_y Ay$ and $A'(x+y)=\lambda_{x+y} A(x+y)$ for
some scalars $\lambda_x$, $\lambda_y$ and $\lambda_{x+y}$.  This
forces $\lambda_x = \lambda_y = \lambda_{x+y}$. So the assertion
follows.
\end{proof}

\begin{lem}\label{Lemma2.6}
Suppose $r$ and $s$ are nonnegative integers with $(r,s)\not=(0,0)$.
Let $X$ be a complex Banach space. If $A, A^\prime \in{\mathcal
B}(X)$ satisfy
$$
\sigma (B^rAB^s+B^sAB^r) =\sigma (B^rA^\prime B^s+B^sA^\prime B^r),
\quad \forall B\in \cI_1(X),
$$
then $A=A^\prime $.
\end{lem}
\begin{proof}
We may suppose that $A'\neq 0$ since it is obvious that $\sigma
(B^rAB^s+B^sAB^r)=\{0\}$ for all rank one idempotents $B$ implies
that $A=0$.

Assume first that $s \geq r > 0$. Then the assumption implies that
$\sigma (BAB)=\sigma (BA^\prime B)$ and hence $f(Ax) = {\rm tr}
(BAB)={\rm tr}(BA^\prime B)=f(A'x)$ for all rank one idempotents
$B=x\otimes f$. By Lemma \ref{lem:dependence}, we see that $A^\prime
=A$.

Assume then that $s > r =0$ and write the rank-one idempotent $B$ in
the form $B=x\otimes f$ with $\left<x,f\right> =1$. Then $AB^s+B^sA
= AB+BA$, and either

(i) $\tr(AB + BA)$ is the sum of the elements in $\sigma(AB+BA)$, or

(ii) $AB+BA$ has rank two and a repeated nonzero eigenvalue so that
$\tr(AB + BA)$ is twice the sum of the elements in $\sigma(AB +
BA)$.

\noindent Therefore, $\tr(AB+BA)=0$ if and only if $\sigma(AB+BA) =
\{0\}$ or $\{\alpha, -\alpha, 0\}$ for some nonzero $\alpha$. Since
$\sigma(AB+BA) = \sigma(A'B+BA')$, we see that $\tr(AB+BA) = 0$ if
and only if $\tr(A'B + BA') = 0$. It follows from Lemma
\ref{lem:dependence} again that $A'=\lambda A$ for some scalar
$\lambda$.  But the spectrum coincidence implies $\lambda =1$.
\end{proof}

As a direct consequence of Lemma \ref{Lemma2.6} and the condition
\eqref{eq:2.1}, we have

\begin{cor}\label{Claim1}
Let $\Phi$ satisfy the hypothesis of Theorem \ref{Theorem2.2}. Then  $\Phi$ is injective, and $\Phi(0)=0$.
\end{cor}

In the following, we present the proof of Theorem \ref{Theorem2.2}
in several steps.

\subsection{The case $\dim X_2 = 1$}

%\begin{lem}\label{Lemma:dim1}
We claim $\dim X_1 =1$. Suppose on contrary that $\dim X_1 \geq 2$.
Let $\Phi(A) = \lambda_A \in \mathbb C$. Then for the rank one
idempotent $B=\left(
                                      \begin{array}{cc}
                                        1 & 0 \\
                                        0 & 0 \\
                                      \end{array}
                                    \right)\oplus 0$ in $\cA_1$ we have by \eqref{eq:2.1} that
$\lambda_B^{r+s+1}=1$.  Moreover,
$$
\sigma(B^s A B^r+ B^r AB^s) = \sigma(2\lambda_A\lambda_B^{r+s}),
\quad\forall A\in \cA_1.
$$
If $r=0$ then $BA + AB = \left(\begin{array}{cc}
                                        2a & b \\
                                        c & 0 \\
                                      \end{array}
                                    \right)\oplus 0$
for any $A =\left(\begin{array}{cc}
                                        a & b \\
                                        c & d \\
                                      \end{array}
                                    \right)\oplus 0$ in $\cA_1$.
In particular, $BA + AB$ can have two distinct   eigenvalues for
some choices of $a,b,c$.  This contradiction forces $\dim X_1 =1$.
If $r>0$ then we will have
$$
\tr(BAB) = \lambda_A\lambda_B^{r+s}, \quad\forall A\in \cA_1.
$$
Thus
$$
\Phi(A)=\lambda_A = \lambda_B\tr(BAB), \quad\forall A\in \cA_1.
$$
Using another rank one idempotent $B'$ in place of $B$ we will have
the same conclusion. Hence,
$$
\lambda_B \tr(BAB) = \lambda_{B'}\tr(B'AB'), \quad\forall A\in
\cA_1.
$$
This is possible only when $\dim X_1 =1$. In both cases, we see that
$\Phi: \mathbb{C}\to\mathbb{C}$ is an algebra isomorphism given by
$\Phi(\alpha) = \lambda\alpha$ with $\lambda^{r+s+1}=1$.

\subsection{The case $\dim X_2 = 2$}

We first claim that $\dim X_1 \geq 2$. Suppose on contrary that $\dim
X=1$.  Write $\Phi(\alpha)=A_\alpha$. By \eqref{eq:2.1},
$$
\sigma(A_\beta^s A_\alpha A_\beta^r + A_\beta^r A_\alpha A_\beta^s)
= \{2\alpha\beta^{r+s}\}, \quad\forall \alpha, \beta\in \mathbb C.
$$
By the surjectivity of $\Phi$, we assume $A_\alpha =
\left(\begin{array}{cc}
                                        1& 0 \\
                                        0 & 2 \\
                                      \end{array}
                                    \right)$.
Then $A_\alpha^{r+s+1}$ has two distinct   eigenvalues $1$ and
$2^{r+s+1}$, a contradiction.

The following lemma verifies Theorem \ref{Theorem2.2} for the case
when $\dim X_2=2$. Indeed, similar arguments can be used to study
the cases when $2 \le \dim X_2 \le \dim X_1 < \infty$. Anyway, we
will use a unified arguments for all the cases when $\dim X_2 \ge 3$ in
the next subsection.

\begin{lem}\label{Lemma:CK}
Let $n\geq 2$ be a cardinal number.  Denote by $\cV_n$  either  a
standard operator algebra on a Banach space of dimension $n$, or the
Jordan algebra of all self-adjoint bounded operators on a Hilbert
space of dimension $n$.  Denote by $M_2$ the algebra of all $2\times
2$ matrices. Let $\Phi: \cV_n \to M_2$ satisfy
\begin{align}\label{eq:ck}
\sigma(B^rAB^s+B^sAB^r) =
\sigma(\Phi(B)^r\Phi(A)\Phi(B)^s+\Phi(B)^s\Phi(A)\Phi(B)^r)
\end{align}
whenever   $A$ and $B$ in $\cV_n$ have rank one. Then $n=2$, and there is
an $m$th root of unity, $\lambda$, and an invertible operator $S$
such that $\Phi$ assumes either the form
$$
\Phi(X) = \lambda S^{-1}XS \quad\text{or}\quad \Phi(X)=\lambda
S^{-1}X^t S.
$$
\end{lem}
\begin{proof}
We first note that $\cV_n$ contains a copy of $\cV_2$. So we can
assume that $\Phi$ is a map from $\cV_2$ into $M_2$. Let $A$ be a
rank one orthogonal projection. Then $B \in \cV_2$ satisfies
$$
[\tr(A^rBA^s+A^sBA^r)]^2  \ne 4 \det(A^rBA^s+A^sBA^r)
$$
if and only if $A^rBA^s+A^sBA^r$ has distinct eigenvalues. Thus, the
set $\cS$ of all such matrices $B$ form an open dense set of
$\cV_2$. Thus, for four linearly independent rank one orthonormal
projections $A_1, A_2, A_3, A_4$, we get a dense set $\cS$ of
matrices $B \in V_2$ such that $A_j^rBA_j^s+A_j^rBA_j^s$ has two
distinct eigenvalues for $j = 1, \dots, 4$. For each $B \in \cS$,
the rank at most two operator $A^rBA^s+A^sBA^r$ has two distinct
eigenvalues, and so is $\Phi(A)^r\Phi(B)\Phi(A)^s +
\Phi(A)^s\Phi(B)\Phi(A)^r$ for all $A \in \{A_1, \dots, A_4\}$ and
$B \in \cS$. It follows that for $m = r+s+1$
\begin{align*}
2\tr(AB) &= 2\tr(A^{m-1}B) = \tr(A^rBA^s+A^sBA^r)\\
&= \tr(\Phi(A)^r\Phi(B)\Phi(A)^s+\Phi(A)^s\Phi(B)\Phi(A)^r) =
2\tr(\Phi(A)^{m-1}\Phi(B))
\end{align*}
for all $A \in \{A_1,\dots,A_4\}$ and $B \in \cS$. For $X = (x_{ij})
\in M_2$, let $v(X) = (x_{11}\ x_{12}\ x_{21}\ x_{22})^t$. Form the
$4\times 4$ matrices
\begin{align*}
R &= [v(A_1) | v(A_2) | v(A_3) | v(A_4)]^t \intertext{and} \hat R &=
[v(\Phi(A_1)^m) | v(\Phi(A_2)^m) | v(\Phi(A_3)^m) |
v(\Phi(A_4)^m)]^t.
\end{align*}
Then
$$R v(B^t) = \hat R v(\Phi(B^t)), \quad \hbox{  for all } B \in \cS.$$
Pick a linearly independent set  $\{B_1, B_2, B_3, B_4\}$ in $\cS$.
If
\begin{align*}
T &= [v(B_1) | v(B_2) | v(B_3) | v(B_4)] \intertext{and} \hat T &=
[v(\Phi(B_1)) | v(\Phi(B_2)) | v(\Phi(B_3)) | v(\Phi(B_4))],
\end{align*}
then
$$R T = \hat R \hat T.$$
Since the left side is the product of two invertible matrices, the
two matrices on the right side are invertible. So, $\hat R^{-1} R
v(B) = v(\Phi(B))$ for all $B \in \cS$. Consider the   linear map
$\hat \Phi: \cV_2 \rightarrow M_2$ such that
$$\hat R^{-1} R v(B) = v(\hat \Phi(B)).$$
Then
$$\sigma(A^rBA^s+A^sBA^r)
= \sigma(\hat \Phi(A)^r\hat \Phi(B)\hat \Phi(A)^s + \hat
\Phi(A)^s\hat \Phi(B) \hat \Phi(A)^r)$$ for all $A, B \in \cS$. By
the continuity of $X \mapsto \sigma(X)$, we see that the set
equality holds for all $A, B \in \cV_2$. Let $A = B$ be a rank one
orthogonal projection. Since $\sigma(A^{m+1}) = \sigma(\hat
\Phi(A)^{m+1})$, we see that $\hat \Phi(A)$ is similar to $\lambda
{\rm diag}(1,0)$ with $\lambda^{m+1} = 1$. By a connectedness
argument, we see that such $\lambda$ is the same for every rank one
orthogonal projection.  Dividing $\Phi$ by $\lambda$, we can assume
$\lambda =1$. By Lemma \ref{Lemma2.6}, we see that $\hat\Phi$ sends
exactly zero to zero. In case $A$ is a rank one square zero matrix,
$\sigma(\hat\Phi(A)^m)=\sigma(A^m)=\{0\}$, and thus $\hat\Phi(A)$ is
also a rank one square zero matrix.

Write every invertible self-adjoint matrix $A$ in $\cV_2$ as a
linear sum of two orthogonal rank one projections. By \eqref{eq:ck},
we see that $\hat\Phi$ sends orthogonal rank one projections to
orthogonal rank one projections. Hence $\hat\Phi(A^2)=\hat\Phi(A)^2$
for all self-adjoint $2\times 2$ matrices. It follows that
$\hat\Phi(AB + BA) = \hat\Phi(A)\hat\Phi(B) +
\hat\Phi(B)\hat\Phi(A)$ for all self-adjoint $2\times 2$ matrices.
If $\cV_2 = M_2$ then $\hat\Phi((A+iB)^2) = \hat\Phi(A^2) +
i\hat\Phi(AB+BA) + \hat\Phi(B^2) = \hat\Phi(A+iB)^2$, whenever $A,B$
are self-adjoint $2\times 2$ matrices. Consequently, $\hat \Phi$ has
the standard form $X \mapsto S^{-1}XS$ or $X \mapsto  S^{-1}X^tS$,
where $S$ is an invertible $2\times 2$ matrix. Note that $\Phi(X) =
\hat\Phi(X)$ for all $X \in \cS$. We may modify $f$ and assume
that $\Phi(X) = X$ for all $X \in \cS$. So, for any $X \in \cV
\setminus \cS$,
$$\sigma(B^rXB^s+B^sXB^r) = \sigma(B^r\Phi(X)B^s + B^s \Phi(X)B^r)$$
for all $B \in \cS$. One can then argue that $\Phi(X) = X$ by Lemma
\ref{Lemma2.6}. Finally, by Corollary \ref{Claim1} we see that
$\Phi$ is injective, and thus $n=2$.
\end{proof}

\subsection{The case $\dim X_2 \geq 3$}

Here are some technical lemmas.

\begin{lem}\label{Lemma2.8}
Let $X$ be a complex Banach space and $A\in{\mathcal B}(X)$. Assume
that $x\otimes f\in{\mathcal B}(X)$ is a rank one idempotent. Then
the at most rank two operator $A(x\otimes f)+(x\otimes f)A$ has
\begin{enumerate}[(1)]
    \item  a nonzero repeated
eigenvalue if and only if $\langle Ax,f\rangle\not= 0$ and $\langle
A^2x,f\rangle=0$;
    \item two distinct nonzero
eigenvalues if and only if $\langle A^2x,f\rangle\not=0$ and
$\langle A^2x,f\rangle\not=\langle Ax,f\rangle^2$.
\end{enumerate}
\end{lem}
\begin{proof}
 (1) Assume that $B=A(x\otimes f)+(x\otimes f)A=Ax\otimes
f+x\otimes A^*f$ has rank two and a nonzero repeated eigenvalue
$\lambda$. Then $\langle Ax,f\rangle=\frac{1}{2}{\rm tr}(A(x\otimes
f)+(x\otimes f)A) =\lambda\not=0$. Furthermore, let $u=Ax-\lambda x$
and $g=A^*f-\lambda f$. Then $\langle x, g\rangle =\langle u,
f\rangle=0$. In a space decomposition with basic vectors $u, x$, the
operator $B$ has a matrix form
$$
B = \left(%
\begin{array}{cc}
  0 & 1 \\
  \langle u,g\rangle & 2\lambda \\
\end{array}%
\right) \oplus\ { 0}.
$$
  Hence, the spectrum of $B$
contains  the zeros of $t^2-2\lambda t-\langle u,g\rangle$, which
gives the repeated eigenvalue $\lambda$ of the operator.  We have
$\langle u,g\rangle =-\lambda^2$. So, $\langle A^2x, f\rangle
=\langle Ax, A^*f\rangle=\lambda^2 +\langle u,g\rangle =0$.

Conversely, if $\langle Ax,f\rangle=\lambda\not= 0$ and $\langle
A^2x,f\rangle=0$, then $Ax=\lambda x+u$ and $A^*f=\lambda f+g$ with
$\langle u, f\rangle=\langle x,g\rangle=0$ and $\langle
u,g\rangle=-\lambda^2$. This implies that $\lambda$ is a repeated
nonzero eigenvalue of $Ax\otimes f+x\otimes A^*f$.

(2) Use the same notations as in the proof of (1). If $A(x\otimes
f)+(x\otimes f)A$ has two distinct nonzero eigenvalues, then, by
(1), $\langle A^2x, f\rangle =\langle Ax, A^*f\rangle=\lambda^2
+\langle u,g\rangle =\langle Ax, f\rangle^2+\langle
u,g\rangle\not=0$ and $\langle u,g\rangle\not=0$. Thus, $\langle
A^2x, f\rangle\not=\langle Ax, f\rangle^2$. The converse is clear.
\end{proof}

\begin{lem}\label{Lemma2.9}
Let $X$ be a complex Banach space of dimension at least two, and let
$A_i\in{\mathcal B}(X)$ with $A_i^2\not=0$, $i=1,2,3$.
Then, the set of rank one idempotent operators $P\in
{\mathcal B}(X)$ satisfying that every $A_iP+PA_i$, $i=1,2,3$, has two
distinct nonzero eigenvalues is dense in the set of all rank one
idempotents in ${\mathcal B}(X)$.
\end{lem}
\begin{proof}
Let $P=x\otimes f$  be a rank one idempotent. By Lemma
\ref{Lemma2.8}, if $AP+PA$ does not have two distinct nonzero
eigenvalues, then $\langle A^2 x,f\rangle =0$ or $\langle
A^2x,f\rangle = \langle Ax,f\rangle ^2$. Let $\varepsilon>0$ be a
small positive number. Assume $\langle A^2 x,f\rangle =0$. If
 $A^2x\not =0$, take $h\in X^*$ such that $\langle
 A^2x,h\rangle\not=0$ and let
 $P_{\varepsilon}=(1+\varepsilon\langle x,h\rangle)^{-1}x\otimes
 (f+\varepsilon h)$; if $A^2x=0$ and there exists $u\in X$ such
 that $\langle A^2u,f\rangle\not=0$, let
 $P_{\varepsilon}=(1+\varepsilon\langle
 u,f\rangle)^{-1}(x+\varepsilon u)\otimes f$; if
$A^2x=0$ and there exists no $u\in X$ such
 that $\langle A^2u,f\rangle\not=0$,
take $u$ and $h$ such that $\langle A^2u,h\rangle\not=0$ and
 let
  $P_{\varepsilon}=\langle
x+{\varepsilon} u,f+{\varepsilon} h\rangle^{-1}(x+{\varepsilon}
u)\otimes (f+{\varepsilon} h)$. If $\langle A^2 x,f\rangle=\langle A
x,f\rangle^2\not=0$, take any $u$ so that $\{x,u\}$ is linearly
independent and $\langle Au,f\rangle\not=0$, and let
$P_\varepsilon=(1+\varepsilon\langle u,f\rangle)^{-1}(x+\varepsilon
u)\otimes f$. In any case, for sufficient small $\varepsilon$, the
rank one idempotent  $P_{\varepsilon}=x_\varepsilon\otimes
f_\varepsilon$ satisfies that  $\langle A^2x_\varepsilon,
f_\varepsilon\rangle\not=0$, $\langle A^2x_\varepsilon
,f_\varepsilon\rangle\not=\langle Ax_\varepsilon
,f_\varepsilon\rangle^2$, $\lim_{\varepsilon\rightarrow
0}\|x_{\varepsilon}-x\|=0$ and $\lim_{\varepsilon\rightarrow
0}\|f_{\varepsilon}-f\|=0$.

For given $A_i$, $i=1,2,3$, in the lemma,  and for any given
positive number $\delta >0$, by Lemma \ref{Lemma2.8}, we have to
show that for any rank one idempotent  $P$ there exists a rank one
idempotent $Q=u\otimes h$ with $\|P-Q\|<\delta$ such that $\langle
A^2_i u,h\rangle\not=0$ and $\langle A^2_i u,h\rangle\not=\langle
A_i u,h\rangle^2$, $i=1,2,3$.

Given $\delta >0$. If the rank one operator $P=x\otimes f$ is such
that   $\langle A^2_1x,f\rangle=0$ or $\langle
A^2_1x,f\rangle=\langle A_1x,f\rangle^2 $, then, by what has been proved in
the previous paragraph, there exists a rank one idempotent
$Q_1=u_1\otimes h_1$ such that $\|P-Q_1\|<\frac{1}{3}\delta$,
$\langle A^2_1u_1, h_1\rangle\not=0$ and $\langle A^2_1u_1,
h_1\rangle\not=\langle A_1u_1, h_1\rangle^2$. If both $\langle
A^2_iu_1, h_1\rangle\not=0$ and $\langle A_i^2u_1,
h_1\rangle\not=\langle A_iu_1, h_1\rangle^2$ hold for $i=2,3$, then
we are done. If, say, $\langle A_2^2u_1, h_1\rangle=0$ or $\langle
A_2^2u_1, h_1\rangle=\langle A_2u_1, h_1\rangle^2$, there exists a
rank one idempotent $Q_2=u_2\otimes h_2$ with
$$\|u_1-u_2\|<\max\left\{ \frac{\delta}{6\|h_1\|},
\frac{1}{4\|A\|^2\|h_1\|}|\langle A^2u_1,h_1\rangle |\right\},$$ and
$$\|h_1-h_2\|<\max\left\{\frac{\delta}{6(\|u_1\|+1)},
\frac{1}{4\|A\|^2(\|u_1\|+1)}|\langle A^2u_1,h_1\rangle |\right\}$$
such that $\langle A_2^2u_2,h_2\rangle\not=0$ and $\langle
A_2^2u_2,h_2\rangle\not=\langle A_2u_2,h_2\rangle^2$. Then
$\|Q_1-Q_2\|<\frac{1}{3}\delta$, $\langle
A_1^2u_2,h_2\rangle\not=0$,  and $\langle
A_1^2u_2,h_2\rangle\not=\langle A_1u_2,h_2\rangle^2$. If $\langle
A_3^2u_2, h_2\rangle\not=0$ and $\langle
A_3^2u_2,h_2\rangle\not=\langle A_3u_2,h_2\rangle^2$, then we are
done since $\|P-Q_2\|<\frac{2}{3}\delta$; if $\langle A_3^2u_2,
h_2\rangle=0$  or $\langle A_3^2u_2,h_2\rangle=\langle
A_3u_2,h_2\rangle^2$, one may repeat the above process and find
$Q_3=u_3\otimes h_3$ such that $\|Q_2-Q_3\|<\frac{1}{3}\delta$,
$\langle A_i^2u_3, h_3\rangle\not=0$ and $\langle A_i^2u_3,
h_3\rangle\not=\langle A_iu_3, h_3\rangle^2$ for all $i=1,2,3$.
Consequently, we get the desired $Q=Q_3$ as $\|P-Q_3\|<\delta$.
\end{proof}

\begin{lem}\label{Lemma2.10}
Let $X$ be a Banach space of dimension at least 2. Let $P, Q$ in
${\mathcal I}_1(X)$ be such that $\sigma(PQ+QP) = \{0\}$. Then $PQ =
0= QP$ if and only if there does not exist $R$ in ${\mathcal
I}_1(X)$ such that $(PR+RP)/2, (QR+RQ)/2 \in {\mathcal I}_1(X)$.
\end{lem}
\begin{proof}
Let $P, Q \in {\mathcal I}_1(X)$  such that $PQ = 0 = QP$. Then
there is a decomposition of $X$ so that $P$ and $Q$ have operator
matrices
$$\diag(1,0) \oplus 0 \qquad \hbox{ and } \qquad \diag(0,1) \oplus 0.$$
Then for any $R \in {\mathcal I}_1(X)$ such that $(PR+RP)/2 \in
{\mathcal I}_1(X)$, the $(1,1)$ entry of the operator matrix of $R$
equals 1, and the off-diagonal part of the first row or the first
column of the operator matrix of $R$ must be zero to ensure that
$PR+RP$ has rank one. Hence, $R$ has operator matrix
$$
\left(\begin{array}{ccc}
1 & * & * \\
0 & 0 & 0 \\
0 & 0 & 0 \end{array}\right)  \qquad \hbox{ or } \qquad
\left(\begin{array}{ccc}
1 & 0 & 0 \\
* & 0 & 0 \\
* & 0 & 0  \end{array}\right).
$$
Similarly, if $(QR+RQ)/2 \in {\mathcal I}_1(X)$, then $R$ has
operator matrix
$$
\left(\begin{array}{ccc}
0 & 0 & 0 \\
* & 1 & * \\
0 & 0 & 0 \end{array}\right)  \qquad \hbox{ or } \qquad
\left(\begin{array}{ccc}
0 & * & 0 \\
0 & 1 & 0 \\
0 & * & 0  \end{array}\right).
$$
Thus, we cannot have $R \in {\mathcal I}_1(X)$ such that both
$(PR+RP)/2, (QR+RQ)/2 \in \cI_1(X)$.

Conversely, suppose $P, Q \in \cI_1(X)$ are such that $\sigma(PQ+QP)
= \{0\}$. If $PQ \ne 0$ or $QP \ne 0$, then there is a decomposition
of $X$ so that $P$ has operator matrix $\diag(1,0) \oplus 0$ and $Q$
has operator matrix
$$\left(\begin{array}{ccc}
0 & 1 & 0 \\
0 & 1 & 0 \\
0 & 0 & 0 \end{array}\right)  \qquad \hbox{ or } \qquad
\left(\begin{array}{ccc}
0 & 0 & 0 \\
1 & 1 & 0 \\
0 & 0 & 0  \end{array}\right).$$ Let $R$ have operator matrix
$$\left(\begin{array}{ccc}
1 & 0 & 0 \\
1 & 0 & 0 \\
0 & 0 & 0 \end{array}\right)  \qquad \hbox{ or } \qquad
\left(\begin{array}{ccc}
1 & 1 & 0 \\
0 & 0 & 0 \\
0 & 0 & 0  \end{array}\right).$$ Then $PR+RP$ has operator matrix
$$\left(\begin{array}{ccc}
2 & 0 & 0 \\
1 & 0 & 0 \\
0 & 0 & 0 \end{array}\right)  \qquad \hbox{ or } \qquad
\left(\begin{array}{ccc}
2 & 1 & 0 \\
0 & 0 & 0 \\
0 & 0 & 0  \end{array}\right),$$ and  $QR+RQ$ has operator matrix
$$\left(\begin{array}{ccc}
1 & 1 & 0 \\
1 & 1 & 0 \\
0 & 0 & 0 \end{array}\right)  \qquad \hbox{ or } \qquad
\left(\begin{array}{ccc}
1 & 1 & 0 \\
1 & 1 & 0 \\
0 & 0 & 0  \end{array}\right).$$ Hence, $(PR+RP)/2, (QR+RQ)/2 \in
\cI_1(X)$.
\end{proof}

For a Banach space $X$ and a ring automorphism $\tau$ of $\IC$, if
an additive map $T: X \rightarrow X$ satisfies $T(\lambda x)=\tau
(\lambda )Tx$ for all complex $\lambda $ and all vectors $x$, we say
that $T$ is $\tau$-linear. The following result can be proved by a
similar argument as the proof of the main result in \cite{Mol1}; see
also \cite[Lemma 3]{BH2} and \cite[Theorem 2.3, 2.4]{Semrl2}.

\begin{lem}\label{Lemma2.11}
Let $X$ and $Y$ be  complex Banach spaces with dimension at least 3.
Let $\Phi : {\mathcal I}_1(X) \rightarrow {\mathcal I}_1(Y)$ be a
bijective map with the property that
$$PQ=QP=0 \quad \hbox{ if and only if } \quad
\Phi(P)\Phi(Q)=\Phi(Q)\Phi(P)=0$$ for all $P,Q$ in ${\mathcal
I}_1(X)$. Then there exists a ring automorphism $\tau$ of ${\mathbb
C}$ such that one of the following cases holds.
\begin{enumerate}[(i)]
    \item There exists a $\tau$-linear transformation
$T:X\rightarrow Y$ satisfying
$$
\Phi(P)=TPT^{-1}\qquad \hbox{ for all } P\in {\mathcal I}_1(X).
$$
    \item There exists a $\tau$-linear transformation
$T:X^*\rightarrow Y$ satisfying
$$
\Phi(P)=TP^*T^{-1}\qquad \hbox{ for all } P\in {\mathcal I}_1(X).
$$
\end{enumerate}
If $X$ is infinite dimensional, the transformation $T$ is an
invertible bounded linear or conjugate linear operator.
\end{lem}

We are now ready to complete the proof of Theorem \ref{Theorem2.2}.
Recall that  $s > r \geq 0$ and $m=r+s+1\geq2$, and we assume from
now on that $X_2$ has dimension at least $3$.

\smallskip\noindent
\begin{proof}[{Proof of Theorem \ref{Theorem2.2}}]
Recall that $\Phi$ satisfies condition \eqref{eq:2.1}.

\smallskip\noindent
{\bf Claim 1.} \it $\Phi$ is injective, and $\Phi(0)=0$. \rm

It is just Corollary \ref{Claim1}.

\smallskip\noindent
{\bf Claim 2.} \it If $A \in \cA_1$ is a nonzero multiple of a rank
one idempotent, then so is $\Phi(A)$. In particular, if $P \in
\cI_1(X_1)$, then $\Phi(P) = \mu R$ such that $R \in \cI_1(X_2)$ and
$\mu^m = 1$. When $s>r>0$, the map $\Phi$ also sends square zero
rank one operators to square zero rank one operators. \rm

Let $A\not = 0$ be a nonzero multiple of an idempotent, say
$A=\alpha P$, where $0\not=\alpha\in{\mathbb C}$ and $P$ in
${\mathcal A}_1$ is a rank one idempotent operator.  For any $D$ in
${\mathcal A}_2$ of rank at most $3$, there is $C$ in ${\mathcal
A}_1$ such that $\Phi(C)=D$. By equation \eqref{eq:2.1} we have
$$\sigma (D^{r}\Phi(A)D^{s}+D^{s}\Phi(A)D^{r})=\sigma
(C^{r}AC^{s}+C^{s}AC^{r}),$$ which contains $0$ and has at most 2
nonzero elements. Putting $B =  A$ in equation \eqref{eq:2.1},  we
have $\sigma (2\Phi(A)^m)=\sigma (2A^m)\not=\{0\}$. Applying Lemma
\ref{Lemma2.3} or Corollary \ref{Corollary2.5}, depending on $s > r
> 0$, or $s > r=0$, we see that $\Phi(A)$ is a nonzero multiple of
rank one idempotent. Thus $\Phi$ preserves nonzero multiples of rank
one idempotents. If $P$ in $\cA_1$ is a rank one idempotent, then
$\Phi(P) = \mu R$, where $R$ in $\cA_2$ is rank one idempotent and
$\mu \in \IC$. Since $\sigma(2P^m) = \sigma(2\Phi(P)^m)$, we see
that $\mu^m = 1$. The last assertion follows from Lemma
\ref{Lemma2.3} and \eqref{eq:2.1}.

Suppose that $s>r>0$.  In this case, $\Phi$ sends rank one operators
to rank one operators by Claim 2. Observe that if $\Phi(x\otimes
f)=y\otimes g$, by \eqref{eq:2.1} we will have
\begin{align}
\left<\Phi(B)^{r+s}y,g\right> &= \left<B^{r+s}x,f\right> \label{eq:1}\\
\left<y,g\right>^{r+s-1}\left<\Phi(B)y,g\right> &=
\left<x,f\right>^{r+s-1}\left<Bx,f\right>, \quad\forall B\in \cA_1
\label{eq:2}. \intertext{Setting $A=B=x\otimes f$, we also have}
\left<y,g\right>^{r+s+1} &=
\left<x,f\right>^{r+s+1}.\label{eq:fx=gy}
\end{align}
With these three conditions \eqref{eq:1}, \eqref{eq:2} and
\eqref{eq:fx=gy} in hand, we can now utilize the proof of
\cite[Theorem 2.5]{hlw} to arrive at the desired assertions of Theorem
\ref{Theorem2.2}.

\smallskip\noindent
\textbf{Conclusion I.}  \emph{From now on, we know that the case $s
> r > 0$ is done.}

However, since we shall  use some arguments below in the next
section, the case $s>r>0$ is still considered until we reach
Conclusion II in the following.

\smallskip\noindent
{\bf Claim 3.}  \it $\Phi(\alpha A)=\alpha \Phi(A) $ holds for all
$A$ in $\cI_1(X_1)$ and $\alpha$ in ${\mathbb C}$. \rm

\smallskip
Denote $\Phi(A)=C$. Then, for any
 $B\in{\mathcal A}_1$, we have
$$
\begin{array}{rl}
& \sigma(\Phi(B)^r\Phi(\alpha A)\Phi(B)^s + \Phi(B)^s\Phi(\alpha A)\Phi(B)^r)\\
=& \sigma (B^r(\alpha A)B^s+B^s(\alpha A)B^r)\\
=&  \sigma (\alpha\Phi(B)^r\Phi(A)\Phi(B)^s+\alpha\Phi(B)^s\Phi(A)\Phi(B)^r)\\
=& \sigma (\Phi(B)^r (\alpha C)\Phi(B)^s+\Phi(B)^s (\alpha
C)\Phi(B)^r)).
\end{array}
$$
Since $\Phi(\cA_1)$ contains $\cI_1(X_2)$, Lemma \ref{Lemma2.6}
implies $\Phi(\alpha A)= \alpha C = \alpha \Phi(A)$.

\smallskip\noindent
{\bf Claim 4.} \emph{Suppose $\Phi(A)$ is a rank one idempotent.
Then $A^2\neq0$.}

\smallskip
In the case $s>r>0$, it follows from Lemma \ref{Lemma2.3} and
\eqref{eq:2.1} that $A$ has rank $1$.  Then by \eqref{eq:2.1} again,
$A$ could not have  zero trace. Thus $A^2\neq 0$.

Next, we shall see that it is impossible to have $A^2=0$ when $s > r
=0$, either. Assuming $A^2=0$ and noting that $A\neq0$, we would
have a nonzero $x$ in $X_1$ such that $\{x, Ax\}$ is linearly
independent.  Let $B=x\otimes f$ be any rank one idempotent on $X_1$
with $\left<Ax,f\right>=1$, and thus $\lambda\Phi(B)=y\otimes g \in
\cI_1(X_2)$ is a rank one idempotent on $X_2$ with some scalar
$\lambda$ such that $\lambda^{m}$=1. If $AB+BA$ is of rank $1$, then
either $\{x, Ax\}$ is linearly dependent or $\{f, A^*f\}$ is
linearly dependent. However, $A^2=0$ would then establish a
contradiction $x=0$ or $f=0$. On the other hand, as its trace
$2\left<Ax,f\right>=2$, the Jordan product $AB+BA$ has exactly rank
$2$. By Lemma \ref{Lemma2.8}(2), we see that $AB + BA$ cannot have
two distinct nonzero eigenvalues. This forces
\begin{align}\label{eq:A2=0}
\sigma(AB+BA)\cup\{0\}=\{0,1\}=\sigma\left(\Phi(A)\Phi(B) +
\Phi(A)\Phi(B)\right)\cup\{0\}.
\end{align}
As $\Phi(A)$ is a rank one idempotent, Lemma \ref{Lemma2.8}(1)
implies that $\Phi(A)\Phi(B) + \Phi(A)\Phi(B)$ cannot have a nonzero
repeated eigenvalue. Therefore, $\Phi(A)\Phi(B) + \Phi(A)\Phi(B)$
has rank $1$. Consequently, $\{y, \Phi(A)y\}$ or $\{g, \Phi(A)^*g\}$
is linearly dependent. Since $\Phi(A)$ is an idempotent, we have
exactly $y=\Phi(A)y$ or $g=\Phi(A)^*g$. Computing trace in
\eqref{eq:A2=0}, we have the absurd  equality
$1=2\lambda\left<y,g\right>=2\lambda$ with $\lambda^m  = 1$.

\smallskip\noindent
{\bf Claim 5.}  \it Let $\Phi(C)=\Phi(A)+\Phi(B)$. If $rs\not=0$,
then $C=A+B$. If $rs=0$, then together with $A^2\neq 0$, $B^2\neq 0$
and $C^2\neq0$, it implies $C=A+B$. \rm

\smallskip

 Let
$W=\Phi(A)$ and $W^\prime=\Phi(B)$. For any rank one idempotent
$P\in{\mathcal A}_1$,  by Claim 2, $Q=\lambda\Phi(P)$ is a rank one
idempotent for some scalar $\lambda$ with $\lambda^m=1$. It follows
from \eqref{eq:2.1} that
\begin{align*}
\sigma(\lambda(Q^{r}(W+W^\prime)Q^{s}+Q^{s}(W+W^\prime)Q^{r}))
&= \sigma (P^{r}CP^{s}+P^{s}CP^{r}),\\
\sigma (\lambda(Q^{r}WQ^{s}+Q^{s}WQ^{r})) &=\sigma
(P^{r}AP^{s}+P^{s}A P^{r}), \intertext{and} \sigma
(\lambda(Q^{r}W^\prime Q^{s}+Q^{s}W^\prime Q^{r})) &=\sigma
(P^{r}BP^{s}+P^{s}BP^{r}).
\end{align*}

If $rs\not=0$, then the traces of the operators in each side of
above equations are the same. This leads to
$$
{\rm tr} (PCP) ={\rm tr}(\lambda Q(W+W^\prime)Q ) ={\rm
tr}(P(A+B)P)$$ for all rank one idempotents $P$ in ${\mathcal A}_1$.
Hence we  have $C=A+B$ by Lemma \ref{Lemma2.6}.

Assume $rs=0$. Then, for those rank one idempotent operators
$P\in{\mathcal A}_1$ such that every one of $CP+PC$, $AP+PA$ and
$BP+PB$ has two distinct nonzero eigenvalues, applying
\eqref{eq:2.1} and then taking trace, we have
\begin{align}\label{eq:2.3}
{\rm tr} (PC) ={\rm tr}(P(A+B)).
\end{align}
 By
assumption, $A$, $B$ and $C$ are non square-zero. Lemma
\ref{Lemma2.9} ensures that \eqref{eq:2.3} holds for a dense set of
rank one idempotents $P$ in $\cA_1$. As a result, $C=A+B$.

\smallskip\noindent
{\bf Claim 6.}  \it There exists a scalar $\lambda $ with $\lambda
^m=1$ such that $\lambda^{-1}\Phi$ sends rank one idempotents to
rank one idempotents. \rm

\smallskip
Let $f$ be nonzero in $X_1^*$. Assume
$\left<x_1,f\right>=\left<x_2,f\right>=0$, and $\Phi(x_1\otimes f) =
\lambda_1 P_1$, $\Phi(x_2\otimes f) = \lambda_2 P_2$, and
$\Phi((\frac{x_1+x_2}{2})\otimes f) = \lambda_3 P_3$ for some rank
one idempotents $P_1, P_2, P_3$ and scalars $\lambda_1, \lambda_2,
\lambda_3$ with $\lambda_1^m=\lambda_2^m=\lambda_3^m=1$. By Claims
3, 4 and 5, we have
$$
2\lambda_3P_3 = \lambda_1 P_1 + \lambda_2 P_2.
$$
Comparing traces, we have
$$
2\lambda_3 = \lambda_1 + \lambda_2.
$$
Since $\lambda_1^m=\lambda_2^m=\lambda^m=1$, we have
$$
\lambda_1=\lambda_2 = \lambda_3.
$$
Denote this common value by $\lambda_f$. Similarly, for any nonzero
$x$ in $X_1$ we will have an $m$th root $\lambda_x$ of unity
depending only on $x$ such that
$$
\Phi(x\otimes f) = \lambda_x Q_{x\otimes f}
$$
for some rank one idempotent $Q_{x\otimes f}$ whenever $f(x)=1$.

Now consider any two  rank one idempotents $x_1\otimes f_1$ and
$x_2\otimes f_2$ in $\cA_1$. We write $x_1\otimes f_1 \sim
x_2\otimes f_2$ if there is a scalar $\lambda$ with $\lambda^m=1$
such that $\lambda\Phi(x_i\otimes f_i)$ is a rank one idempotent for
$i=1,2$. In   case $\alpha=\left<x_1,f_2\right>\neq 0$, we see that
$$
x_1\otimes f_1 \sim x_1\otimes \frac{f_2}{\alpha} =
\frac{x_1}{\alpha}\otimes f_2 \sim x_2\otimes f_2.
$$
In case $\left<x_1,f_2\right>=\left<x_2,f_1\right>=0$, we also have
$$
x_1\otimes f_1 \sim (x_1+x_2)\otimes f_1 \sim (x_1+x_2)\otimes f_2
\sim x_2\otimes f_2.
$$

\smallskip\noindent
\textbf{Conclusion II.} \it By Claim 6, without loss of generality,
we assume that  $\Phi$ preserves rank one idempotents. By Conclusion
I, it suffices to deal with the case $s>r=0$ in the sequel .\rm

\smallskip\noindent
{\bf Claim 7.}  \it If $\Phi(A)\in\cA_2$ is a rank one idempotent,
then $A\in\cA_1$ is a rank one idempotent. \rm

\smallskip

Suppose $\Phi(A)$ is a rank one idempotent. If $A$ is of rank one,
then Claims 1 and 3 ensure that $A$ is a rank one idempotent. Now we
suppose $A$ has rank at least $2$, and we want to derive a
contradiction. Note that $A^2\neq 0$ by Claim 4.

\smallskip\noindent
\textsc{Case 1.} Suppose there is an $x$ in $X_1$ such that $\{x,
Ax, A^2 x\}$ is linearly independent.  Let $f$ in $X_1^*$ be such
that $\left<x,f\right> =\left<Ax, f\right> =1$, but $\left<A^2x,
f\right>\neq 0$ or $1$.  Lemma \ref{Lemma2.8}(2) ensures that
$A(x\otimes f) + (x\otimes f)A$ has $2$ distinct nonzero
eigenvalues, and so has $\Phi(A)(y\otimes g) + (y\otimes g)\Phi(A)$
by \eqref{eq:2.1}, where $y\otimes g=\Phi(x\otimes f)$ is a rank one
idempotent.  Comparing traces, we have $\left<\Phi(A)y, g\right> =
\left<Ax,f\right>=1$.  This contradicts to Lemma \ref{Lemma2.8}(2),
however.

\smallskip\noindent
\textsc{Case 2.} Suppose $\{x, Ax, A^2x\}$ is  linearly dependent
for all $x$ in $X_1$.  Hence, by Kaplansky's Lemma (\cite{K1954,
A1988}) there are  scalars $a,b,c$, not all zero, such that $aA^2 +
bA + cI =0$.

\smallskip
\textsc{Subcase 2a.} If $A$ has rank $2$ then $A$ has nonzero
eigenvalues $\alpha_1, \alpha_2$ (maybe equal). With respect to a
suitable space decomposition, we can assume
$$
A = \left(
      \begin{array}{ccc}
        \alpha_1 & 0 & 0 \\
        0 & \alpha_2 & 0 \\
        0 & 0 & 0 \\
      \end{array}
    \right)
    \quad\text{or}\quad
A = \left(
      \begin{array}{ccc}
        \alpha_1 & 1 & 0 \\
        0 & \alpha_1 & 0 \\
        0 & 0 & 0 \\
      \end{array}
    \right).
$$
Then
$$
A=\alpha_1 e_1\otimes e_1 + \alpha_2 e_2\otimes e_2
\quad\text{or}\quad A= \alpha_1 e_1\otimes e_1 +
\alpha_1(\frac{e_1}{\alpha_1} +  e_2)\otimes e_2.
$$
By Claims 3 and 5, and Conclusion II, the rank one idempotent
\begin{align*}
\Phi(A)
&=\Phi(\alpha_1 e_1\otimes e_1 + \alpha_2 e_2\otimes e_2)\\
&= \alpha_1 \Phi(e_1\otimes e_1) + \alpha_2 \Phi(e_2\otimes e_2)\\
&= \alpha_1 y_1\otimes g_1 + \alpha_2 y_2\otimes g_2,
\end{align*}
in the first case with rank one idempotents $y_1\otimes g_1 =
\Phi(e_1\otimes e_1)$ and $y_2\otimes g_2 = \Phi(e_2\otimes e_2)$.
Observing ranks, we see that $\{y_1, y_2\}$ or $\{g_1,g_2\}$ is
linearly dependent. On the other hand, as $\left<e_1,
e_2\right>\left<e_2, e_1\right>=0$ we see by \eqref{eq:2.1} that
$\left<y_2,g_1\right>\left<y_1,g_2\right> = 0$. This eventually
gives the contradiction $1=\left<y_1,g_1\right>\left<y_2,g_2\right>
= 0$. The second case is similar.

\smallskip
\textsc{Subcase 2b.} Assume $A$ has rank at least $3$.  Since $A$ is
quadratic, each Jordan block of $A$ has order either $1$ or $2$.
Consider the case
$$
A= \left(
      \begin{array}{cccc}
        \alpha_1 & 1 & 0 & 0\\
        0 & \alpha_1 & 0 &0\\
        0 & 0 & \alpha_2 & 0 \\
        0 & 0 & 0 & * \\
      \end{array}
    \right).
$$
Here the nonzero eigenvalues $\alpha_1, \alpha_2$ of $A$ can be
equal. Then
$$
Ae_1 = \alpha e_1,\quad Ae_2 = e_1 + \alpha_1
e_2\quad\text{and}\quad Ae_3 = \alpha_2 e_3.
$$
Observe
\begin{align*}
A(e_1\otimes e_1) + (e_1\otimes e_1)A &= e_1\otimes (2\alpha_1 e_1+e_2),\\
A(e_2\otimes e_2) + (e_2\otimes e_2)A &= (e_1 + 2\alpha_1
e_2)\otimes e_2, \intertext{and} A(e_3\otimes e_3) + (e_3\otimes
e_3)A &= 2\alpha_2 e_3\otimes e_3.
\end{align*}
Consider the rank one idempotents
 $\Phi(A) = y\otimes g$, and $\Phi(e_i\otimes e_i) = y_i\otimes g_i$ for $i=1,2,3$.
By \eqref{eq:2.1}, we see that
\begin{align*}
\sigma((y\otimes g)(y_i\otimes g_i) + (y_i\otimes g_i)(y\otimes
g))\cup\{0\} = \{0, 2\alpha_1\} \text{ or } \{0, 2\alpha_2\},
\quad\text{for } i=1,2,3.
\end{align*}
  In particular, by Lemma \ref{Lemma2.8}(1),
\begin{align}\label{eq:notzero}
\left<y_i, g\right>\left<y,g_i\right> = \alpha_1 \text{ or }
\alpha_2, \text{ is not zero, for } i=1,2,3.
\end{align}
But as $\left<e_i, e_j\right>\left<e_j,e_i\right>=0$, we have
$$
\left<y_i,g_j\right>\left<y_j,g_i\right>=0 \quad\text{whenever
$i\neq j$}.
$$
On the other hand, Lemma \ref{Lemma2.8}(1) and \eqref{eq:notzero}
force all $(y\otimes g)(y_i\otimes g_i) + (y_i\otimes g_i)(y\otimes
g)$ have rank one. Consequently, $\{y_i, y\}$ or $\{g,g_i\}$ is
linearly dependent for each $i=1,2,3$. Eventually, we might have two
of $y_1,y_2, y_3$ are linearly dependent, or two of $g_1,g_2,g_3$
are linearly dependent.  Suppose $y_1,y_2$ are dependent. Since
$g_1(y_1)=g_2(y_2)=1$, we see  that
$\left<y_1,g_2\right>\left<y_2,g_1\right>=0$, which is absurd. We
shall reach   other contradictions similarly for other possible
situations. Analogously, we can also derive a contradiction when we
are dealing with the case
$$
A= \left(
      \begin{array}{cccc}
        \alpha_1 & 0 & 0 & 0\\
        0 & \alpha_1 & 0 &0\\
        0 & 0 & \alpha_2 & 0 \\
        0 & 0 & 0 & * \\
      \end{array}
    \right)
    \quad\text{or}\quad
A= \left(
      \begin{array}{ccccc}
        \alpha_1 & 1 & 0 & 0& 0\\
        0 & \alpha_1 & 0 & 0& 0\\
        0 & 0 & \alpha_2 & 1 & 0 \\
        0 & 0 & 0 & \alpha_2 & 0 \\
        0 & 0 & 0 & 0 & * \\
      \end{array}
      \right).
$$
This completes the verification of Claim 7.

\smallskip\noindent
{\bf Claim 8.} \it One of the following statements is true.

{\rm (i)} There exists a bounded invertible linear operator $T:
X_1\rightarrow X_2$ such that
$$ \Phi(x\otimes f)=T(x\otimes f)T^{-1} \quad\mbox{\rm for
all } x\in X_1, f\in X_1^*\mbox{ }{\rm with}\mbox{ }\langle
x,f\rangle =1.$$

{\rm (ii)} There exists a bounded invertible linear operator $T:
X_1^*\rightarrow X_2$ such that
$$
\Phi(x\otimes f)=  T(x\otimes f)^*T^{-1} \quad\mbox{\rm for all }
x\in X_1, f\in X_1^*\mbox{ }{\rm with}\mbox{ }\langle x,f\rangle =1.
$$
\rm

Since $\Phi$ preserves rank one idempotents in both directions, by
use of Lemma \ref{Lemma2.10}, it is easily checked that $P, Q \in
\cI_1(X)$ satisfy $PQ=0=QP$ if and only if $\Phi(P)\Phi(Q) = 0 =
\Phi(Q)\Phi(P)$. Thus we can apply Lemma \ref{Lemma2.11} to conclude
that (i) or (ii) holds, but with $T$ a $\tau$-linear for some ring
automorphism $\tau$ of $\IC$.

 Next we prove that $\tau$ is the
identity and hence $T$ is linear.  For any $\alpha \in \IC \setminus
\{1,0\}$, let $A$ and $B$ have operator matrices
$$\left(\begin{array}{cc}
1 & \alpha - 1 \\
0 & 0
\end{array}\right) \oplus 0 \qquad \hbox{ and } \qquad
\left(\begin{array}{cc}
1 & 0 \\
1 & 0
\end{array}\right) \oplus 0.$$
Then $AB+BA$ has two distinct nonzero eigenvalues summing up to
$2\alpha$. Since
$$\begin{array}{rl}
& \sigma(AB+BA) =  \sigma(\Phi(A)\Phi(B) + \Phi(B)\Phi(A)) \\
=&\sigma (T(AB+BA)T^{-1}) = \{\tau (\xi ):\xi\in\sigma (AB+BA)\},
\end{array}$$
we see that
$$2\alpha =
\tr(AB+BA)= \tr(\Phi(A)\Phi(B) + \Phi(B)\Phi(A)) =
\tr(T(AB+BA)T^{-1})=2\tau (\alpha ).$$ Hence $\tau(\alpha) = \alpha$
for any $\alpha \in \IC$. It follows that $T$ is an invertible
bounded linear operator.

\smallskip\noindent
{\bf Claim 9.} \it  $\Phi$ has the form in Theorem \ref{Theorem2.2}.
\rm

Suppose (i) in Claim 8 holds. Let $A\in{\mathcal A}_1$ be arbitrary.
For any $x\in X_1$ and $f\in X_1^*$ with $\langle x,f\rangle =1$,
the condition \eqref{eq:2.1} ensures that
$$\begin{array}{rl}
& \sigma ((T^{-1}\Phi(A)T)(x\otimes f)^{s}
+(x\otimes f)^{s}(T^{-1}\Phi(A)T))\\
=& \sigma (T[T^{-1}\Phi(A)T(x\otimes f)^{s}
+(x\otimes f)^{s}T^{-1}\Phi(A)T]T^{-1})  \\
= &\sigma (A(x\otimes f)^{s}+(x\otimes f)^{s}A).\end{array}$$ Hence,
by Lemma \ref{Lemma2.6}, we have $$\Phi(A)= TAT^{-1}$$ for all $A$
in ${\mathcal A}_1$, that is, $\Phi$ has the form (1) in the
theorem.

Similarly, one can show that $\Phi$ has the form (2) if (ii) of
Claim 8 holds.
\end{proof}

\section{Generalized Jordan product spectrum preserving maps of self-adjoint operators}

Let $H$ be a complex Hilbert space and ${\mathcal S}(H)$ be the real
linear space of all self-adjoint operators in ${\mathcal B}(H)$.
Note that ${\mathcal S}(H)$ is a Jordan algebra. In this section we
solve the problems discussed previously for maps on ${\mathcal
S}(H)$. Our results refine those in \cite{CLS}.

\begin{thm}\label{Theorem3.1}
For $i=1,2$, let $H_i$ be a complex
 Hilbert space,
 and ${\mathcal S}(H_i)$ be the Jordan algebra of all bounded
 self-adjoint operators on
 $H_i$. Consider the product $T_1\circ \cdots \circ T_k$
defined in Definition \ref{Definition2.1}. Suppose
 $\Phi :{\mathcal S}(H_1)\rightarrow{\mathcal S}(H_2)$
satisfies
\begin{align}\label{eq:3.2}
\sigma(\Phi(A_1)\circ \Phi(A_2)\circ\cdots \circ\Phi(A_k))
   =\sigma(A_1\circ A_2\circ\cdots \circ A_k),
\end{align}
whenever any one of the $A_i$'s has rank at most one. Suppose
further that the range of $\phi$ contains all self-adjoint operators
of rank at most $3$. Then there exist a scalar $\xi$ in $\{-1,1\}$
with $\xi^{m}=1$ and a unitary operator $U: H_1\rightarrow H_2$ such
that either
\begin{align*}
\Phi(A) &=\xi UAU^{*} \quad \text{for all $A$ in
$\mathcal{S}(H_1)$}, \intertext{or} \Phi(A) &=\xi UA^tU^{*} \quad
\text{for all $A$ in ${\mathcal S}(H_1)$},
\end{align*}
 where $A^t$ is the transpose of $A$ for
an arbitrarily but fixed orthonormal basis.
\end{thm}

To prove Theorem \ref{Theorem3.1}, it is important to characterize
rank one operators in terms of the general Jordan products of
self-adjoint operators. We have the following lemma.

\begin{lem}\label{Lemma3.4}
Suppose $s > r \geq0$ is a pair of nonnegative integers. Let $H$ be
a Hilbert space of dimension at least three, and let
$0\not=A\in{\mathcal S}(H)$. Then the following statements are
equivalent.
\begin{enumerate}[(a)]
    \item $A$ has rank one.

    \item For any $B\in{\mathcal S}(H)$, $\sigma(B^rAB^s+B^sAB^r)$
contains $0$ and at most two nonzero elements.

    \item There does not exist $B\in{\mathcal S}(H)$ of rank at
most three such that $B^rAB^s+B^sAB^r$ has rank at most three and
$\sigma(B^rAB^s+B^sAB^r)$ contains three distinct nonzero elements.
\end{enumerate}
\end{lem}
\begin{proof}
The implications (a) $\Rightarrow$ (b) $\Rightarrow$ (c) are clear.
To prove (c) $\Rightarrow$ (a), we consider the contrapositive.
Suppose (a) does not hold. Assume $rs \ne 0$. If $A$ has rank at
least 3, then there are vectors $x_1, x_2, x_3$ such that $\{Ax_1,
Ax_2, Ax_3\}$ is linearly independent.  Extend an orthonormal basis
for $[x_1, x_2, x_3, Ax_1, Ax_2, Ax_3]$ to an orthonormal basis for
$H$. Then the operator matrix of $A$ with respect to this basis has
the form
$$\left(\begin{array}{cc}
A_{11} & A_{12} \\
A_{12}^* & A_{22}
\end{array}\right),$$
where $A_{11}=A_{11}^*$ is the compression of $A$ on the subspace
$[x_1, x_2, x_3, Ax_1, Ax_2, Ax_3]$. By \cite[Lemma 2.3]{hlw}, we
can choose an orthonormal basis for $[x_1, x_2, x_3, Ax_1, Ax_2,
Ax_3]$ so that the leading $3\times 3$ matrix of $A_{11}$ equals
$\diag(a_1, a_2, a_3)$ for some nonzero scalars $a_1, a_2, a_3$. Now
construct $B$ so that the operator matrix of $B$ using the same
basis as that of $A$ equals $\diag(1, b_2, b_3) \oplus 0\oplus 0$ so
that $a_1, a_2b_2^{r+s}, a_3b_3^{r+s}$ are distinct nonzero numbers.
Then $B^rAB^s+B^sAB^r$  has rank $3$ with three distinct nonzero
eigenvalues.

Next, suppose $A$ has rank 2. Choosing a suitable basis, we may
assume that $A$ has operator matrix $\diag(a, b,0) \oplus 0$.
Construct $B$ with operator matrix $[d] \oplus B_1 \oplus 0$, where
$$B_1 =
\left(\begin{array}{cc}1&1\\ 1&-1\end{array}\right)
\left(\begin{array}{cc}2&0\\0 &1\end{array}\right)
\left(\begin{array}{cc}1&1\\
1&-1\end{array}\right) =2\left(\begin{array}{cc}1&1\cr
1&1\end{array}\right) + \left(\begin{array}{cc}1&-1\cr -1
&1\end{array}\right).
$$
Compute
$$
B_1^k =2^{k-1}\left[ 2^k\left(\begin{array}{cc}1&1\\
1&1\end{array}\right) + \left(\begin{array}{cc}1&-1\\ -1
&1\end{array}\right)\right], \qquad k = 1, 2, \dots.$$ Now, if
$\gamma = 2^r$ and $\delta = 2^s$ then
$$
B_1^r\left(\begin{array}{cc}1 & 0 \\ 0 & 0
\end{array}\right)B_1^s + B_1^s\left(\begin{array}{cc}1 & 0 \\ 0 & 0
\end{array}\right)B_1^r = 2^{r+s-1}
\left(\begin{array}{cc}(\gamma+1)(\delta+1) & \gamma\delta -1  \\
\gamma\delta-1 & (\gamma-1)(\delta-1)\end{array}\right)$$ has
determinant $-4^{r+s-1}(\gamma-\delta)^2 < 0$. So, it has a positive
and a negative eigenvalue, say, $\mu$ and $\nu$. Thus, we can
choose $d$ so that $B^rAB^s + B^sAB^r$ has three nonzero distinct
nonzero eigenvalues: $2ad^{r+s}, b\mu, b\nu$.

Next, suppose $s>r=0$. If $A$ has rank $2$, then $A$ has an operator
matrix of the form $\diag(a_1, a_2,0) \oplus 0$ for some nonzero
real numbers $a_1, a_2$. Let $b > 0$ be such that $2b^sa_1 \ne
a_2(1/2 \pm 1/\sqrt{2})$. Suppose $B \in \cS(H)$ is such that $B$
and $AB^s + B^sA$ have operator matrices
$$\left(\begin{array}{ccc}
b & 0 & 0  \\
0 & 1/2 & 1/2 \\
0 & 1/2 & 1/2
\end{array}\right) \oplus 0 \qquad \hbox{ and } \qquad
\left(\begin{array}{ccc}
2a_1b^s & 0 & 0  \\
0 & a_2 & a_2/2 \\
0 & a_2/2 & 0
\end{array}\right) \oplus 0.$$
Then $AB^s + B^sA$  has rank $3$ with three distinct nonzero
eigenvalues $2b^sa_1, a_2(1/2 + 1/\sqrt{2})$ and $a_2(1/2 -
1/\sqrt{2})$.

Now, suppose $A$ has rank at least 3. If $A = \lambda I$, then let
$B$ have operator matrix $\diag(1,2,3) \oplus 0$ with respect to
some orthonormal basis for $H$. Then $B$ has rank $3$ and $AB^s +
B^sA$ has rank $3$ with three distinct nonzero eigenvalues $\lambda,
2^{s}\lambda, 3^{s}\lambda$. So, assume $A$ is non-scalar. Thus,
there is a unit vector $x_1 \in H$ such that $Ax_1 = a_1 x_1 + a_2
x_2$ with $a_1\not=0$ and $a_2 > 0$, where $x_2$ is a unit vector in
$[x_1]^\perp$. Let $A x_2 = b_1 x_1 + b_2 x_2 + b_3 x_3$ with $b_3
\ge 0$, where $x_3$ is a unit vector in $[x_1, x_2]^\perp$. We
consider two cases.

\textsc{Case 1.} If $b_3 > 0$, then the operator matrix of the
self-adjoint operator $A$ with respect to an orthonormal basis with
$\{x_1, x_2, x_3\}$ as the first three vectors has the form \iffalse
$$\left(\begin{array}{cccc}
a_{1} & b_{1} & * & * \\
a_{2} & b_{2} & * & * \\
0     & b_{3} & * & * \\
0     & 0     & * & *
\end{array}\right).$$
Since $A$ is self-adjoint, the operator matrix  actually has the
form \fi
$$\left(\begin{array}{cccc}
a_{1} & a_{2} & 0 & 0 \\
a_{2} & b_{2} & b_3  & 0 \\
0     & b_{3} & * & * \\
0     & 0     & * & *
\end{array}\right).$$
Let $B$ have operator matrix $I_2 \oplus 0$. Then $AB^s + B^sA$ has
an operator matrix of the form $C_1 \oplus 0$, where
$$
C_1 = \left(\begin{array}{ccc} 2a_1 & 2a_2 & 0 \\ 2a_{2} & 2b_2 &
b_3  \\ 0 & b_3 & 0
\end{array}\right).$$
Note that $\det(C_1) = -2a_1b_3^2 \ne 0$, and $C_1 - \lambda I$ has
rank at least two for any eigenvalue $\lambda$ as the $2\times 2$
submatrix at the right top corner is always invertible.  So, $C_1$
is invertible  and has three distinct nonzero eigenvalues. Hence,
$AB^s+B^sA$ has rank $3$ with three  distinct nonzero eigenvalues.

\textsc{Case 2.} Suppose $b_3 = 0$. Then $[x_1, x_2]$ is an
invariant subspace of $A$. Since $A$ has rank at least 3, there is a
unit vector $x_3$ in $H$ such that $Ax_3\not=0$ and $Ax_3\in\{x_1,
x_2\}^\perp$.

\textsc{Subcase 2a.} If $[x_1, x_2, x_3]$ is an invariant subspace
of $A$, then with respect to an orthonormal basis for $[x_1, x_2,
x_3]$ and its orthonormal complement, $A$ has operator matrix $A_1
\oplus A_2$, where $A_1$ in $M_3$ has rank at least $2$. If $A_1$
has rank $3$, we may assume that $A_1 = \diag(a_1, a_2, a_3)$. We
can choose $B$ with operator matrix $\diag(b_1, b_2, b_3) \oplus 0$
for some suitable $b_1, b_2, b_3$ so that $AB^s +B^sA$ has rank $3$
with three distinct nonzero eigenvalues $2a_1 b_1^s, 2a_2 b_2^s,
2a_3 b_3^s$. If $A_1$ has rank 2, we may assume that $A_1 =
\diag(a_1, a_2,0)$  and continue exactly as when $A$
has rank $2$. Then choose $B$ with operator matrix
$$\left(\begin{array}{ccc}
b & 0 & 0  \\
0 & 1/2 & 1/2 \\
0 & 1/2 & 1/2
\end{array}\right) \oplus 0$$
so that $2b^sa_1 \ne a_2(1/2 \pm 1/\sqrt{2})$. Then $AB^s + B^sA$
has rank $3$ with three distinct nonzero eigenvalues  $2b^sa_1 \ne
a_2(1/2 \pm 1/\sqrt{2})$.

\textsc{Subcase 2b.} Suppose $Ax_3 = c_3 x_3 + c_4 x_4$ so that $c_4
> 0$ and  $\{x_1, x_2, x_3, x_4\}$ is an orthonormal set in $H$.
If $Ax_4 = d_3 x_3 + d_4 x_4 + d_5 x_5$ so that $\{x_3, x_4, x_5\}$
is an orthonormal set in $H$ and $d_5 > 0$, then we are back to Case
1 with $(x_1, x_2)$ replaced by $(x_3,x_4)$. We thus assume that
$[x_1, x_2, x_3, x_4]$ is an invariant subspace of $A$. With respect
to an orthonormal basis for $[x_1, x_2, x_3, x_4]$ and its
orthonormal complement, $A$ has operator matrix $A_3 \oplus A_4$,
where $A_3 \in M_4$ is self-adjoint and has rank at least 2. We may
assume that $A_3$ is in diagonal form with at least two nonzero
diagonal entries. Using a similar argument as in Subcase 2A, we get
the desired conclusion.
\end{proof}

\smallskip\noindent
\begin{proof}[{Proof of Theorem \ref{Theorem3.1}}]
Assume that $\Phi$ satisfies \eqref{eq:3.2s}. Let
$$
r = \min\{p-1,m-p\} \quad \hbox{ and } \quad s = \max\{p-1,m-p\}.
$$
In particular, $r+s = m-1$.
It suffices to prove a
special case of Theorem \ref{Theorem3.1}, as that Theorem
\ref{Theorem2.2} to Theorem \ref{Theorem2} in last section.
More precisely, we assume the condition
\begin{align}\label{eq:3.2s}
\sigma(\Phi(B)^r \Phi(A)\Phi(B)^s + \Phi(B)^s\Phi(A)\Phi(B)^r) =
 \sigma(B^rAB^s + B^sAB^r)
 \end{align}
 holds whenever $A$ or $B$ in ${\mathcal S}(H_1)$ has rank at most one.
The case
$s=r$ has been done in \cite{hlw}. Hence, we assume $s>r\geq 0$.
Arguing similarly as in the beginning of the proof of Theorem
\ref{Theorem2.2}, we can verify the case $\dim H_2 \leq 2$.
Therefore, we assume the dimension of the Hilbert space $H_2$ is at
least three in the sequel.

\smallskip\noindent
{\bf Claim 1.} \it $\Phi$ is injective, and $\Phi(0)=0$. \rm

This works out similarly as in Corollary \ref{Claim1}.

\smallskip\noindent
{\bf Claim 2.} \it $\Phi$ sends rank one self-adjoint operators to
rank one self-adjoint operators.

\smallskip \rm
This follows from \eqref{eq:3.2s} and Lemma \ref{Lemma3.4}. Indeed,
every rank one self-adjoint operator has the form $\pm\, x\otimes
x$. So, $\Phi(x\otimes x)=\lambda_xy_x\otimes y_x$ for some
$\lambda_x\in\{-1,1\}$ and $y_x\in H_2$. Since
$$\{2\|x\|^{2m},0\}=\sigma (2(x\otimes x)^m)=\sigma(2\Phi(x\otimes x)^m)
=\{2\lambda_x^m\|y_x\|^{2m},0\},$$ we see that $\lambda_x$ is an
$m$th root of the unity and $\|y_x\|=\|x\|$.

\smallskip\noindent
{\bf Claim 3.} \it $\Phi$ is real homogeneous; and if $\Phi(C) =
\Phi(A) +\Phi(B)$ then $C=A+B$.
 Moreover, there is a fixed $\lambda$, being either $+1$ or
$-1$, such that for every $x$ in $H_1$ we have $\Phi(x\otimes
x)=\lambda y_x\otimes y_x$ with $\|y_x\|=\|x\|$. \rm

The assertions follow from arguments similar to, and a bit easier
than, that in Claims 3, 5 and 6 in the proof of Theorem
\ref{Theorem2.2} in last section.

\smallskip\noindent
{\bf Claim 4.} \it $\Phi$ has the form stated in the theorem. \rm

Let $x,x'$ be two nonzero vectors in $H_1$, and $x\otimes x$ and
$x'\otimes x'$ be the associated rank one self-adjoint operators,
respectively. By \eqref{eq:3.2s}, and Lemma \ref{Lemma2.8} when
$s>r=0$, we see that
$$
\tr (\Phi(x\otimes x)\Phi(x'\otimes x')) = \tr ((x\otimes
x)(x'\otimes x')),
$$
or
$$
\left<\lambda_{x}y_x,\lambda_{x'}y_{x'}\right> = \left<x,x'\right>.
$$
This gives
$$
|\left<y_x,y_{x'}\right>| = |\left<x,x'\right>|, \quad  \text{for
all nonzero } x,x'\in H_1.
$$
If follows from the Wigner's Theorem \cite{Gyory} that there exist a
modular one function $\xi: H_1 \to \mathbb C$  and a  linear or
conjugate linear isometry $U: H_1\rightarrow H_2$ such that
$$
y_x = \xi(x)Ux, \quad\forall x \in H_1.
$$
By Claim 3, we see that all $\xi(x)$ equal a constant
$\xi\in\{-1,+1\}$, and
$$
\Phi(x\otimes x)=\xi Ux\otimes Ux \quad\text{for all rank one
projection $x\otimes x$ on $H_1$}.
$$
Moreover, \eqref{eq:3.2s} ensures that  $\xi^m=1$. Because the range
of $\Phi$ contains all rank one self-adjoint operators, by
\eqref{eq:3.2s} we can see that $U$ has dense range, and thus $U$ is
a unitary or a conjugate unitary operator.

In general, for any $A$ in ${\mathcal S}(H_1)$, let $A_{i_p}=A$ and
$A_{i_q}=x\otimes x$ with $\|x\|=1$ if $q\not = p$, and substitute
them into \eqref{eq:3.2s}. Since both $A$ and $\Phi(A)$ are
self-adjoint, we see that
$$
\begin{array}{rl} &\sigma (\xi^{m-1}
((x\otimes x)^rU^*\Phi(A)U(x\otimes x)^s+(x\otimes
x)^sU^*\Phi(A)U(x\otimes x)^r))\\=&\sigma ((x\otimes x)^rA(x\otimes
x)^s+(x\otimes x)^sA(x\otimes x)^r).
\end{array}
$$
By Lemma \ref{Lemma2.8} and comparing traces,  we get $\Phi(A)=\xi
UAU^*$ for all $A$ in ${\mathcal S}(H_1)$. If $U$ is a conjugate
unitary, take an orthonormal basis $\{ e_j\}$ of $H_1$ and define a
conjugate unitary $J:H_1\rightarrow H_1$ by $J:\sum_j\xi_je_j\mapsto
\sum_j\bar{\xi}_je_j$ and let $V=UJ$. Then $V$ is unitary and
$JA^*J=A^{t}$. Thus, $\Phi(A)=VA^tV^*$ for all $A$ in ${\mathcal
S}(H_1)$.
\end{proof}

 \end{document}

\title[Addendum to ``Maps preserving the spectrum'']
{Addendum to ``Maps preserving the spectrum of generalized Jordan product of operators''}

\author{Jinchuan Hou, Chi-Kwong Li, and Ngai-Ching Wong}
\address[Hou]{Department of Mathematics, Taiyuan University of Technology,
Taiyuan, 030024, P. R. of China} \email{houjinchuan@tyut.edu.cn}
\address[Li] {Department of Mathematics, The College of William
\& Mary, Williamsburg, VA 13185, USA. {\rm Li is an honorary
professor of the University of Hong Kong}} \email{ckli@math.wm.edu}

\address[Wong]{Department of Applied Mathematics, National Sun Yat-sen
  University, Kaohsiung, 80424,
Taiwan. } \email{wong@math.nsysu.edu.tw}

\date{April 20, 2010}

\thanks{{\it 2002 Mathematical Subject Classification.}
Primary 47B49; 47A12}
\thanks{{\it Key words and phrases.} Standard operator algebra, spectral
functions, Jordan triple-products of operators, skew products of
operators, nonlinear-preserving problems}

\maketitle

Regarding our paper \cite{hlw-laa2010},
Jianlian Cui pointed out that some arguments in the proof of Theorem 3.1
are not entirely clear and accurate. Here we supply some details.

\bigskip

\noindent\textbf{Theorem 3.1.} (\cite{hlw-laa2010})\ {\it
Suppose a map $\Phi:\cA_1\to\cA_2$ between standard operator
algebras satisfies
\begin{align}\label{eq:2.1}
\sigma( \Phi(B)^r\Phi(A)\Phi(B)^s +  \Phi(B)^s\Phi(A)\Phi(B)^r)
   =\sigma(B^rAB^s+B^sAB^r),\tag{3.1}
\end{align}
whenever $A$ or $B$ has rank at most one. Suppose also that the
range of $\Phi$ contains all operators in $\cA_2$ of rank at most
$3$. Then one of the following two assertions  holds
with  $m = r+s+1$.
\begin{enumerate}[(1)]
    \item There exist a scalar $\lambda$ with $\lambda^{m}=1$ and
an invertible operator $T$  in $\mathcal{B}(X_1,X_2)$ such that
$$
\Phi(A)=\lambda TAT^{-1} \quad \text{for all $A$ in
$\mathcal{A}_1$}.
$$
    \item The spaces $X_1$ and $X_2$ are reflexive, and there exist
a scalar $\lambda$ with $\lambda^{m}=1$ and an invertible operator
$T \in\mathcal{B}(X_1^{\ast},X_2)$ such that
$$
\Phi(A)=\lambda TA^{\ast}T^{-1}\quad\text{for all $A$ in
$\mathcal{A}_1$}.
$$
\end{enumerate}
}

\medskip

\begin{proof}[Some modifications in the proof of Theorem 3.1.]
First, remove the paragraph
``In the case $s>r>0$, \ldots. Thus $A^2\neq 0$''
 after \textbf{Claim 4}, as we do not need this in the proof.

\medskip\noindent
{\bf Claim 6.}  \it There exists a scalar $\lambda $ with $\lambda
^m=1$ such that $\lambda^{-1}\Phi$ sends rank one idempotents to
rank one idempotents. \rm

\smallskip
First line of the proof
should be
``Let $f$ be nonzero in $X_1^*$. Assume $\langle x_1, f \rangle =\langle x_2, f \rangle= 1$''.
If $rs \neq 0$, the original proof works.  In case $s > r=0$
the proof can continue as follow. ``By Claim 2,
$\Phi(x_1\otimes f) =
\lambda_1 y_1\otimes g_1$ and $\Phi(x_2\otimes f) = \lambda_2 y_2\otimes g_2$, where
$g_1(y_1)=g_2(y_2) =1$ and $\lambda_1^{s+1}=\lambda_2^{s+1} =1$.
Using the spectrum equation (3.1) we have
\begin{align*}
  &\ \sigma(\lambda_1^{s}\lambda_2 (g_1(y_2)\, y_1\otimes g_2 + g_2(y_1)\, y_2\otimes g_1))\\
= &\ \sigma(\lambda_2^{s}\lambda_1 (g_2(y_1)\, y_2\otimes g_1 + g_1(y_2)\,  y_1\otimes g_2))\\
= &\ \sigma((x_1\otimes f)(x_2\otimes f) + (x_2\otimes f)(x_1\otimes f))\\
= &\ \{0, 2\}.
\end{align*}
By Lemma 3.6(1) and computing traces, we would have
$$
\lambda_1^s\lambda_2 g_1(y_2)g_2(y_1) = \lambda_1\lambda_2^s g_1(y_2)g_2(y_1) = 1.
$$
In particular, $\lambda_1^2 =\lambda_2^2$.  Suppose $\lambda_1 = - \lambda_2$.  Then
we have
$g_1(y_2)g_2(y_1) = -1$, and  by Lemma 3.6(2), we will get a contradiction.  So $\lambda_1 = \lambda_2$.''
At this point we can go back to the original proof again.

\medskip\noindent
{\bf Claim 7.} \it If $\Phi(A) \in \cA_2$ is a rank one idempotent,
then $A \in \cA_1$ is a rank one idempotent.
\rm

Assume that $\Phi(A)$ is a rank one idempotent. Suppose $A$ is not a
rank one idempotent.

\medskip\noindent
{\sc Case 1.} As is.

\medskip\noindent
{\sc Case 2.} Suppose $\{x, Ax, A^2x\}$ is always linearly dependent.
Then by Kaplansky's result, $A$ is a quadratic operator, i.e.,
there is $a, b$ such that $(A-aI)(A-bI) = 0$. Then
with respect to a suitable space decomposition of $X_1$, $A$ has an operator
matrix of the form
$$\begin{pmatrix} aI & T \cr 0 & bI \cr\end{pmatrix},$$
where $T$ may be assumed to be zero if $a \ne b$.

%{\bf Should I cite the reference I found on the internet?}

If $A$ has rank one, then $A$ has operator matrix $[a] \oplus 0$.
By Claim 3, we have $a = 1$ and we are done.
So, assume that $A$ has rank at least two.
If $a, b$ are distinct, we may assume that
$a\ne 0$ and the null space of
$A-aI$ has dimension at least 2 as $X_1$ has dimension at least 3.
Moreover, we may assume that $A$ has operator matrix $aI \oplus bI$.
Suppose $B_1, B_2, B_3$ have operator matrices
$$
\begin{pmatrix}1 & 0 \cr 0 & 0 \cr\end{pmatrix}\oplus 0, \quad
\begin{pmatrix}0 & 0 \cr 0 & 1 \cr\end{pmatrix}\oplus 0, \quad
\hbox{ and } \quad
\frac{1}{2}\begin{pmatrix}1 & 1 \cr 1 & 1 \cr\end{pmatrix}\oplus 0,$$
we see that
$\sigma(B_j^sA+AB_j^s) = \{2a,0\}$
and $\Phi(B_j)$ is a rank one idempotent for $j = 1,2,3$.

Note that  all $\Phi(A)$, $\Phi(B_1)$, $\Phi(B_2)$
and $\Phi(B_3)$ are rank one idempotents, and thus we can
find  a subspace $V$ of $X_2$ of dimension at most 4 such that
in a suitable space decomposition
$X_2 = V \oplus V'$
these operators can be written as direct sums of $4\times 4$ matrices and zero.
So we might assume that $X_2$ has dimension at most 4 in the following discussion.

Assume that
$\Phi(A)$ and $\Phi(B_1)$
have operator matrices
$$[1] \oplus 0 \quad \hbox{ and } \quad
\begin{pmatrix}
b_{11} & b_{12} & b_{13} \cr
b_{21} & b_{22} & b_{23} \cr
0 & 0 & 0 \end{pmatrix} \oplus 0,$$
respectively, such that $b_{11} + b_{22} = 1$ and
$b_{11}b_{22}=b_{12}b_{21}$.
Then
$\Phi(B_1)^s\Phi(A) + \Phi(A)\Phi(B_1)^s$ has operator matrix
$$\begin{pmatrix}
2b_{11} & b_{12} & b_{13} \cr
b_{21} & 0 & 0 \cr
0 & 0 & 0 \end{pmatrix} \oplus 0\quad
\hbox{ with spectrum equal to } \quad
\sigma(B_1A^s+A^sB_1) =\{2a,0\}.$$
It follows that
(i) $b_{11} = a$, $b_{12}b_{21} = 0$, or
(ii) $b_{11} = 2a$ and $b_{12}b_{21} = -4a^2$.
If (ii) holds, then $b_{22} = b_{12}b_{21}/b_{11} = -2a$. But then
$\Phi(B_1)$ has trace zero, which is impossible.
Thus, (i) holds with $a = 1$. By Corollary 3.4, $\Phi(B_1) \ne \Phi(A)$,
we may thus assume that with a space decomposition
$X_2 = \Span \{x\} \oplus V_1$,
$\Phi(B_1)$ has operator matrix
$$M_1 = \begin{pmatrix}
1 & 1 & 0\cr
0 & 0 & 0\cr
0 & 0 & 0 \cr\end{pmatrix}
\quad
\hbox{ or } \quad
M_2 = \begin{pmatrix}
1 & 0 & 0\cr
1 & 0 & 0\cr
0 & 0 & 0 \cr\end{pmatrix}.$$
Next, we turn to $\Phi(B_j)$ for $j = 2, 3$.
Similarly, with a space decomposition $X_2 = \Span \{x\} \oplus V_j$,
we can assume that the operator matrix of $\Phi(A)$ has the form $[1]\oplus 0$,
and $\Phi(B_j)$ have the form $M_1$ or $M_2$.
Now, let $S_j = [1] \oplus T_j$, where $T_j$ is an operator
of changing a basis for $V_1$ to a basis for $V_j$.
Then we see that the operator matrix of $\Phi(B_j)$ with
respect to the space decomposition $\Span \{x\} \oplus V_1$ has the form
$S_j^{-1}M_1S_j$ or $S_j^{-1}M_2S_j$.
Hence, with respect to the space decomposition $\Span \{x\} \oplus V_1$,
$A$ has operator matrix $[1] \oplus 0$, and for $j = 1, 2, 3$,
$\Phi(B_j)$ has operator matrix
\begin{equation} \label{form}
\begin{pmatrix}
1 & X_j\cr
0 & 0 \cr
\end{pmatrix}
\quad
\hbox{ or } \quad
\begin{pmatrix}
1 & 0 \cr
Y_j & 0 \cr
\end{pmatrix},
\end{equation}
where $X_j, Y_j$ are rank one operators.
But then there are two distinct elements
$j, k\in \{1, 2, 3\}$ such that both $\Phi(B_j)$ and $\Phi(B_k)$
have operator matrices in upper or lower triangular form.
It follows that
$$\sigma(\Phi(B_j)^s\Phi(B_k) + \Phi(B_k)\Phi(B_j)^s) = \{2, 0\}
\ne \sigma(B_j^sB_k+B_kB_j^s),$$
which is a contradiction.

Next, suppose $a = b$. Then $A$ has an operator matrix
$\begin{pmatrix} aI_1 & T \cr 0 & aI_2\cr\end{pmatrix}$ for some
operator $T$.

By claim 4, we see that $A^2\neq 0$ and thus
$a \ne 0$.  Let $B_1, B_2, B_3$
have operator matrices $\diag(1,0,0) \oplus 0, \ \diag(0,1,0)\oplus 0,
\ \diag(0,0,1) \oplus 0$. Then
$\sigma(\Phi(A)\Phi(B_j)^s+\Phi(B_j)^s\Phi(A))
 =\sigma(AB_j^s+B_j^sA) =  \{2a,0\}$. Using argument as before, we see that
$\Phi(B_j)$ has operator matrix of the form (\ref{form}), and there are
distinct elements $j,k \in \{1, 2, 3\}$ such that
$\sigma(\Phi(B_j)^s\Phi(B_k) + \Phi(B_k)\Phi(B_j)^s) = \{2, 0\} \ne
\sigma(B_j^sB_k+B_kB_j^s)$, which is a contradiction.
\qed

\medskip
To prove Claim 8, Lemma 3.8 in the paper \cite{hlw-laa2010} should be replaced by
the following.

\medskip\noindent
{\bf Lemma 3.8.} {\it Suppose ${\rm dim} X \ge 3$.
Let $P, Q \in \cI_1(X)$. Then $PQ = 0 = QP$ if
and only if there is $B \in \cB(X)$, which can be chosen to have rank
2, such that $\sigma(PB+BP) = \{2,0\}$, $\sigma(QB+QB) = \{-2,0\}$, and
$\sigma(BR+RB) = \{0\}$ whenever $R \in \cI_1(X)$
satisfies $\sigma(PR+RP) = \sigma(QR+RQ) = \{0\}$.}

\it Proof. \rm Suppose $P, Q$ in $\cI_1(X)$ satisfy $PQ=0=QP$.
Then there is a space decomposition  for $X$ such that $P$ and $Q$ have operator
matrices
$$\begin{pmatrix}1 & 0 \cr 0 & 0 \cr \end{pmatrix} \oplus 0
\qquad \hbox{ and } \qquad
\begin{pmatrix}0 & 0 \cr 0 & 1 \cr \end{pmatrix} \oplus 0.$$
Using the same space decomposition, let $B$ have operator matrix
$\begin{pmatrix}1&0\cr 0 & -1\end{pmatrix} \oplus 0$.
Then $B$ has rank 2 such that $\sigma(PB+BP) = \{2,0\}$,
$\sigma(QB+BQ) = \{-2,0\}$.
Consider any $R$ in $\cI_1(X)$ such that
$\sigma(PR+RP) = \sigma(QR+RQ) = \{0\}$.
Using the same space decomposition as $P$ and $Q$,
we assume that $R$ has operator matrix
$$\begin{pmatrix}R_{11} & R_{12} \cr R_{21} & R_{22} \cr \end{pmatrix}$$
where $R_{11}$ is a $2\times 2$ matrix .
Since $\sigma(PR+RP) = \sigma(QR+RQ) = \{0\}$, the $(1,1)$ and $(2,2)$ entry
of $R_{11}$ are both zero.
Thus, $R_{22}$ has trace one and rank one. We may then assume that
$R_{22}$ has operator matrix $[1] \oplus 0$. As a result, we may assume that
the operator matrix of $R$ has the form  $\hat R \oplus 0$,
where $\hat R$ or $\hat R^t$ has one of the following forms:
$$
\begin{pmatrix} 0&0&* \cr 0&0&*\cr 0&0&1\cr\end{pmatrix},
\quad \hbox{ or }
\quad \begin{pmatrix} 0&a& b\cr 0&0&0\cr 0&c&1\cr\end{pmatrix} \hbox{ with }
a = bc.$$
Consequently, $\sigma(BR+RB) = \{0\}$.

Conversely, suppose $P, Q \in \cI_1(X)$ such that  $PQ \ne 0$ or
$QP \ne 0$. Then there is a space decomposition for $X$ such that
$P$ has operator matrix
$\begin{pmatrix}1 & 0 \cr 0 & 0 \cr \end{pmatrix}
\oplus 0$ and $Q$ has operator matrix
$$\begin{pmatrix}0 & 0 \cr 1 & 1 \cr\end{pmatrix} \oplus 0
\qquad \hbox{ or  } \qquad
\begin{pmatrix}0 & 1 \cr 0 & 1 \cr \end{pmatrix} \oplus 0.$$
We assume that the former case holds. The proof for the other case
is similar.
Suppose there is a $B$ in $\cB(X)$ such that
$\sigma(BR+RB) = \{0\}$ whenever $R$ in $\cI_1(X)$ satisfies
$\sigma(PR+RP) = \sigma(QR+RQ) = \{0\}$.
Using the same space decomposition as $P$ and $Q$,
we assume that $B$ has operator matrix
$$\begin{pmatrix}B_{11} & B_{12} \cr B_{21} & B_{22} \cr\end{pmatrix}$$
where $B_{11}$ is a $2\times 2$ matrix .

First, we claim that $B_{22} = 0$. If not, we may assume that the
$(1,1)$ entry of $B_{22}$ is nonzero. If
$R$ has operator matrix
$$
\begin{pmatrix}0 & 0 & 0\cr 0 & 0 &0\cr 0 & 0 & 1 \end{pmatrix} \oplus 0,
$$
we see that
$\sigma(PR+RP) = \sigma(QR+RQ) = \{0\} \ne \sigma(BR+RB)$.

Next, we claim that $B_{12} = 0$. If it is not true, we can find
a suitable space decomposition  for $X$ such that the rank one
block $B_{12}$ has the form
$\begin{pmatrix} 1 & 0 \cr
0 & T  \cr\end{pmatrix},$
where the last column is vacuous if $\dim X = 3$, and $T$ has rank
zero or one.
But then if
$R \in \cI_1(X)$ has operator matrix
$$
\begin{pmatrix}0 & 0 & 0 \cr 0 & 0 & 0 \cr 1 & 0 & 1 \cr\end{pmatrix} \oplus 0,
$$
we have $\tr(BR+RB) = 2$ so that $\sigma(BR+RB) \ne \{0\}$.
Similarly, we can show that $B_{21} = 0$.

Now, consider $B_{11} = \begin{pmatrix} b_{11} & b_{12}\cr b_{21} & b_{22}\cr\end{pmatrix}$.
Let $R$ have operator matrix
$\begin{pmatrix}0 & 0 & 0 \cr
                1 & 0 & 1 \cr
                1 & 0 & 1 \cr \end{pmatrix} \oplus 0.$
We see that $\sigma(PR+RP) = \sigma(QR+RQ) = \{0\}$.
Because $BR + RB$  has operator matrix
$$
\left(\begin{pmatrix}b_{12} & 0 & b_{12} \cr
                b_{22} & 0 & b_{22} \cr
                0 & 0 & 0 \cr \end{pmatrix} +
\begin{pmatrix}0 & 0 & 0 \cr
               b_{11} & b_{12} & 0 \cr
               b_{11} & b_{12} & 0 \cr \end{pmatrix} \right)\oplus 0$$
with trace $2b_{12}$ and
$\sigma(BR+RB) = \{0\}$, we see that $b_{12}= 0$.
Since, $\sigma(PB+BP) = \{2,0\}$ and $\sigma(QB+BQ) = \{-2,0\}$,
it follows that $(b_{11}, b_{22}) =(1,-1)$.
Finally, for $R$ with operator matrix
$\begin{pmatrix}0 & -1 \cr 0 & 1 \cr\end{pmatrix} \oplus 0$,
we have
$\sigma(PR+RP) = \sigma(QR+RQ) = \{0\}$. But then
$BR+RB$ has operator matrix
$$ \left(
\begin{pmatrix} 0 & -1 \cr 0 & -1-b_{21} \cr\end{pmatrix} +
\begin{pmatrix} -b_{21} & 1 \cr b_{21} & -1\cr\end{pmatrix} \right)
\oplus 0
= \begin{pmatrix} -b_{21} & 0 \cr b_{21} & -2 - b_{21} \cr\end{pmatrix}
\oplus 0,$$
which  cannot be a nilpotent.
\end{proof}